\documentclass[10pt]{amsart}

\usepackage[colorlinks]{hyperref}
\usepackage{color,graphicx,shortvrb}
\usepackage[utf8,latin 1]{inputenc}

\newtheorem{theorem}{Theorem}[section]

\newtheorem{corollary}[theorem]{Corollary}

\newtheorem{definition}[theorem]{Definition}
\newtheorem{lemma}[theorem]{Lemma}

\newtheorem{proposition}[theorem]{Proposition}
\newtheorem{remark}[theorem]{Remark}

\usepackage[active]{srcltx} 

\usepackage{stackrel}
\usepackage{amssymb}

\def\J#1#2#3{ \left\{ #1,#2,#3 \right\} }

\def\11{\mathbf{1}}

\usepackage{enumerate}
\usepackage{amsmath}
\usepackage{amssymb}
\usepackage{mathabx}

\begin{document}

\title{A projection--less approach to Rickart Jordan structures}
\author[J.J. Garc{\' e}s]{Jorge J. Garc{\' e}s}
\address[J.J. Garc{\' e}s]{Departamento de Matem{\'a}tica Aplicada, ETSIDI, Universidad
Polt{\' e}cnica de Madrid, Madrid}
\email{j.garces@upm.es}

\author[L. Li]{Lei Li}
\address[L. Li]{School of Mathematical Sciences and LPMC,
	Nankai University, Tianjin 300071,  China}
\email{leilee@nankai.edu.cn}

\author[A.M. Peralta]{Antonio M. Peralta}
\address[A.M. Peralta]{Departamento de An{\'a}lisis Matem{\'a}tico, Facultad de
	Ciencias, Universidad de Granada, 18071 Granada, Spain.}
\email{aperalta@ugr.es}

\author[H.M. Tahlawi]{Haifa M. Tahlawi}
\address[H.M. Tahlawi]{Department of Mathematics, College of Science, King Saud University, P.O.Box 2455-5, Riyadh-11451, Kingdom of Saudi Arabia.}
\email{htahlawi@ksu.edu.sa}

\dedicatory{To the memory of Professor C.M. Edwards with admiration, affection, and respect}



\date{}

\begin{abstract} The main goal of this paper is to introduce and explore an appropriate notion of weakly Rickart JB$^*$-triples. We introduce weakly order Rickart JB$^*$-triples, and we show that a C$^*$-algebra $A$ is a weakly (order) Rickart JB$^*$-triple precisely when it is a weakly Rickart C$^*$-algebra. We also prove that the Peirce-2 subspace associated with a tripotent in a weakly order Rickart JB$^*$-triple is a Rickart JB$^*$-algebra in the sense of Ayupov and Arzikulov. By extending a classical property of Rickart C$^*$-algebras, we prove that every weakly order Rickart JB$^*$-triple is generated by its tripotents. 
\end{abstract}

\maketitle
\thispagestyle{empty}


\section{Introduction and preliminaries}

The reference \cite{Rick46} is the founding work of the fruitful theory of Rickart and Baer C$^*$-algebras. C. E. Rickart \cite{Rick46} stated that \emph{``Our general purpose is to study the structure of a B$^*$-algebra in terms of its projections. Such a study of course demands the existence of many projections .... a B$^*$-algebra is defined to be a $B_p^*$ -algebra} (now called a \emph{Rickart C$^*$-algebra}) 
\emph{provided it contains, in a certain sense, ``sufficiently many'' projections.''} The chosen notion was built around left and right annihilators. For each nonempty subset $S$ of an associative ring $A$, the \emph{right-} and \emph{left-annihilator} of $S$ are defined by $$R(S)=\{ x\in A : sx = 0 \hbox{ for all } s\in S\}$$ and $$L(S)=\{x\in A : x s = 0 \hbox{ for all } s\in S\},$$ respectively. If $A$ is an associative $^*$-ring, a \emph{projection} $p$ in $A$ will be a self-adjoint ($p^* = p$) idempotent ($p^2 = p$). A \emph{Rickart $^*$-ring} is an associative $^*$-ring $A$ such that, for each $a\in A$, $R(\{a\}) = p A$ for a (unique) projection $p$ (see \cite[\S 3, Definition 2]{BerBaerStarRings}). In such a case we have $L(\{a\}) = \left(R(\{a^*\})\right)^* = \left( q A\right)^* = A q$ for a suitable projection $q$. A \emph{Rickart C$^*$-algebra} is a C$^*$-algebra which is also a Rickart $^*$-ring (cf. \cite[\S 3, Definition 3]{BerBaerStarRings} and the original work by Rickart \cite{Rick46}). Each Rickart $^*$-ring has a unity element and its involution is \emph{proper}, i.e., $x x^* =0 \Rightarrow x=0$ (see \cite[\S 3, Proposition 2]{BerBaerStarRings}). The projections of a Rickart C$^*$-algebra form a lattice which is not necessarily complete (cf. \cite[\S 3, Proposition 7 and Example 2]{BerBaerStarRings}). A C$^*$-algebra $A$ is called weakly Rickart if for each $x\in A$ there exists an \emph{annihilating right projection} (briefly, ARP) of $x$, that is, a projection $p$ satisfying $x p  = x,$ and $x y = 0$ implies $p y = 0$. Let us observe that annihilating left projections (ALP) are similarly defined. The ARP and ALP of each element $x$ are uniquely determined by $x$, and we shall denote them by $RP(x)$ and $LP(x)$, respectively. Every unital weakly Rickart C$^*$-algebra is a Rickart C$^*$-algebra, since for each $x\in A$ we have $R(\{x\}) = (\mathbf{1}-RP(x)) A.$ Rickart proved in \cite[Theorem 2.10]{Rick46} that every Rickart C$^*$-algebra is generated by its projections.\smallskip

As seen before, the definition of a Rickart $^*$-ring is given in terms of the annihilators of singletons. When singletons are replaced by general subsets we find the notion of Baer $^*$-ring. Concretely, a \emph{Baer $^*$-ring} is an associative $^*$-ring $A$ such that, for every nonempty
subset $S \subset A$ we have $R(S) = p A$ for a suitable projection $p$ in $A$ (see \cite[\S 4, Definition 1]{BerBaerStarRings}). Baer $^*$-rings are precisely those Rickart $^*$-rings whose projections form a complete lattice, equivalently, every orthogonal family of projections has a supremum (cf. \cite[\S 4, Proposition 1]{BerBaerStarRings}). As introduced in the pioneering works of Kaplansky \cite{Kap51, Kap52, Kap53}, an \emph{AW$^*$-algebra} is a C$^*$-algebra that is a Baer $^*$-ring (see \cite[\S 4, Definition 2]{BerBaerStarRings}).\smallskip

Since for each element $a$ in a C$^*$-algebra $A$ we have $R(\{a\}) =  R(\{a^* a\})$, in the definition of Rickart C$^*$-algebra we can restrict our attention to the right-annihilators of positive elements. Similarly in the definition of AW$^*$-algebras we can consider right-annihilators of sets of the form $\{ a^* a : a\in S\}$, where $S$ is any subset of the C$^*$-algebra under study.\smallskip

Each \emph{von Neumann algebra} (i.e., a $^*$-subalgebra of $B(H)$ whose bicommutant coincides with itself, or equivalently, by Sakai's theorem \cite{Sak56CharacteWStar}, a C$^*$-algebra which is also a dual Banach space) is an AW$^*$-algebra \cite[\S 4, Proposition 9]{BerBaerStarRings}. After  Sakai's theorem, von Neumann algebras are also known as W$^*$-algebras. Though AW$^*$-algebras were actually introduced with the aim of finding an algebraic characterization of von Neumann or W$^*$-algebras, it was soon shown by Dixmier that there exist commutative AW$^*$-algebras which cannot be represented as von Neumann algebras (see \cite{Dix51} or \cite[\S 7, Exercises 2, 3]{BerBaerStarRings}). Wright found in \cite{Wright76} examples of monotone complete factors which are not von Neumann algebras. The reader has probably realised that we take the references \cite{Rick46, Kap51, BerBaerStarRings, SaitoWrightBook} as the basic bibliography on Rickart and AW$^*$-algebras.\smallskip

In the list of problems and future directions in \cite[page 144]{RodPa91Overwolfach},  A. Rodr{\'i}guez-Palacios somehow anticipated and suggested the study of Rickart Jordan algebras as those Jordan algebras for which \emph{``the annihilator of every element in Zelmanov sense is generated by an idempotent''} (see subsection \ref{subsec: background} for the basic theory on Jordan algebras). We have to wait until 2016 to find the first study on Rickart and Baer Jordan algebras by Sh. A. Ayupov and F. N. Arzikulov (see \cite{AyuArzi2016Rickart}). The (\emph{outer}) \emph{quadratic annihilator} of a subset $\mathcal{S}$ in a Jordan algebra $M$ --with product $\circ$-- is defined as the set $$\hbox{Ann} (\mathcal{S}) = \mathcal{S}^{\perp_q} :=\{ a\in M : U_a (s) = 2 (a \circ s)
\circ a - (a\circ a) \circ s = 0, \ \ \ \forall s\in S \}.$$ A Jordan algebra $M$ is called a \emph{Rickart Jordan algebra} if for each element $a\in M^2$ there exists an idempotent $e \in M$ such that $\{a\}^{\perp_q}= U_e(M),$ where $U_e (x) := 2(e\circ x) \circ e - e^2\circ x$. If in the definition of Rickart Jordan algebra, the sets given by a single element $a\in M^2$ are replaced by arbitrary subsets $\mathcal{S}\subset M^2,$ we get the notion of \emph{Baer Jordan algebra} (cf. \cite{AyuArzi2016Rickart}).\smallskip

Rickart and Baer Jordan algebras are appropriate notions for JB-algebras, where we have projections and positive elements. It is shown by Ayupov and Arzikulov that for each C$^*$-algebra $A$, its self-adjoint part, $A_{sa},$ is a Rickart (respectively, Baer) Jordan algebra if and only if $A$ is a Rickart (respectively, Baer) C$^*$-algebra \cite{AyuArzi2016Rickart, AyuArzi2019Rickart}. A Rickart (respectively, Baer) JB$^*$-algebra is a JB$^*$-algebra $M$ whose self-adjoint part is a Rickart (respectively, Baer) JB-algebra.\smallskip

The original aim in Rickart's studies was completed in the case of JB-algebras by F.N. Arzikulov who proved that a  JB-algebra $N$ is a Baer Jordan algebra if and only if $N$ satisfies the following properties:
\begin{enumerate}[$(1)$]\item Every subset of pairwise orthogonal projections in a partially ordered set of projections has a least upper bound in this set;
\item  Every maximal strongly associative subalgebra of $N$ is generated by its projections (see \cite[Theorem 2.1]{Arzi98}).
\end{enumerate}

The available notions of Rickart and Baer Jordan algebras have a strong dependence on quadratic annihilators, projections and positive elements. However, if we are interested in developing these notions in more general Jordan structures, like JB$^*$-triples, where projections and positive elements do not make any sense, we need an alternative approach. This is the main goal of this paper.\smallskip

Section \ref{sec: Rickart Cstar-algebras} is devoted to revisit the main results on Rickart and weakly Rickart C$^*$-algebras with the aim of finding a characterization which can be stated without appealing to projections and positive elements. We shall show (see Propositions \ref{p sufficient and necessary conditions in terms of inner ideals without order} and \ref{p sufficient and necessary conditions in terms of inner ideals without order but local order}) that by mixing and extending a characterization due to G.K. Pedersen in \cite{PedSAW} with key contributions by P. Ara and D. Goldstein \cite{Ara89,AraGold1993,Gold95}, the following characterizations hold for every C$^*$-algebra $A$: \begin{enumerate}[$(a)$]\item $A$ is a weakly Rickart C$^*$-algebra if, and only if, any of the following statements holds:
\begin{enumerate}[$(1)$]
\item Given $x\in A$ and an inner ideal $J\subseteq A$ which is orthogonal to the inner ideal $I=A(x)$ of $A$ generated by $x$, there exists a partial isometry $e$ in $A$ such that $I\subseteq A_2(e)$ and $J\subseteq A_0(e)$.
\item Given $x\in A$ and an inner ideal $J\subseteq A$ with $I = A(x)\perp J $, there exists a partial isometry $e$ in $A$ such that $I\subseteq A_2(e),$ $e^*e =RP(x)$, $e e^* = LP(x),$ $x$ is a positive element in the C$^*$-algebra $(A_2(e), \bullet_e, *_e),$ $A(x)$ is a C$^*$-subalgebra of the latter C$^*$-algebra and $J\subseteq A_0(e)$.
\end{enumerate}
\item $A$ is a Rickart C$^*$-algebra if, and only if, $A$ is unital and for each $x\in A$ and each inner ideal $J\subseteq A$ which is orthogonal to $I=A(x)$, there exists a partial isometry $e$ in $A$ such that $I\subseteq A_2(e)$ and $J\subseteq A_0(e)$.
\end{enumerate} 

The advantage of the previous characterizations (especially the one in $(a) (1)$) relies on their independence of projections and positive elements, and can be therefore extended to wider settings. Before further extensions, in section \ref{subsec: algebraic Jordan Rickart and Baer} we explore the notions of weakly Rickart and SAJBW-algebras, both in terms of projections and positive elements. For example, a JB-algebra $N$ is called a {weakly Rickart JB-algebra} if for each element $a\in N^+$ there exists a projection $p \in N$ such that $p\circ a = a,$ and for each $z\in N$ with  $U_z(a) =0$ we have $p\circ z =0$. In Proposition \ref{p sufficient and necessary conditions in terms of inner ideals} we establish several characterizations of Baer or AJBW$^*$-algebras, (weakly) Rickart JB$^*$-algebras and SAJBW$^*$-algebras in terms of hereditary JB$^*$-subalgebras. After several technical conclusions in the line of classical results, we arrive to our main goal of section \ref{subsec: algebraic Jordan Rickart and Baer} in Theorem \ref{t weakly Rickart JBstar algebras contain an abundant set of projections}, where it is proved that every weakly Rickart JB$^*$-algebra is generated by its projections.\smallskip 

In section \ref{sec: JB*-triples} we introduce several definitions of Rickart, weakly Rickart and weakly order Rickart JB$^*$-triples. We show that, thanks to the characterization of the corresponding notions for C$^*$-algebras presented in section \ref{sec: Rickart Cstar-algebras}, the new definitions coincide with the classical notions in the setting of C$^*$-algebras. Special interest is received by weakly order Rickart JB$^*$-triples. This new notion agrees with the concept of Rickart C$^*$-algebra in the C$^*$- setting. A weakly order-Rickart JB$^*$-triple $E$ is a JB$^*$-triple satisfying that for each $x\in E,$ if we write $E(x)$ for the inner ideal of $E$ generated by $x$, then for each inner ideal $J\subseteq E$ with $I = E(x)\perp J $, there exists a tripotent $e$ in $E$ such that $x$ is positive in $E_2(e)$, and $J\subseteq E_0(e)$.\smallskip

We prove in Propostion \ref{p woR implies wR for each Peirce-2} that if $E$ is a weakly order Rickart JB$^*$-triple and $e\in E$ is a tripotent, then the Peirce-2 subspace $E_2(e)$ is a Rickart JB$^*$-algebra. This allows us to conclude that every weakly order Rickart JB$^*$-triple is generated by its tripotents (see Theorem \ref{t woR JBstriples are generated by its tripotents}).\smallskip

Finally, in section \ref{sec: von Neumann regularity} we explore the connections with von Neumann regularity, by showing that each inner ideal $I$ of a weakly order Rickart JB$^*$-triple $E$ contains a dense subset of von Neumann regular elements (cf. Theorem \ref {t inner ideals of woR JBstar triples contain a norm dense subset of regular elements}).

\subsection{Background and basic definitions}\label{subsec: background}

This subsection is aimed to provide a basic compendium on the Jordan structures studied in this note. The reader will find some brief historical introduction, definitions, notions and basic references. These contents are not really required to follow section \ref{sec: Rickart Cstar-algebras}, which has been written to be accessible with tools of C$^*$-algebras.\smallskip 

The early contributions by Jordan, von Neumann and Wigner in the decade of 1930s led to the idea of employing non-associative structures, specially Jordan algebras, in quantum mechanics (see the interesting monograph \cite{LiebRuHens} for a fantastic historical overview). A real or complex Jordan algebra is a non-necessarily associative algebra $M$ whose product (denoted by $\circ$) is commutative and satisfies the \emph{Jordan-identity}: \begin{equation}\label{eq Jordan identity algebras} (a \circ b)\circ a^2 = a\circ (b \circ a^2)\ \ (a,b\in M).
\end{equation} The Jordan algebra $M$ is called unital if there exists a unit element $\11$ in $M$ such that $\11\circ a = a$ for all $a\in M$. Jordan algebras are power associative, that is, a subalgebra generated by a single element is  
associative. In other words, for each $a\in M$ define $a^0:=\11$ if $M$ is unital, $a^1 = a$ and $a^{n+1} = 
a\circ  a^n$ ($n\geq 1$). Then $a^{n+m} = a^n \circ a^m$  for all natural numbers $m$ and $n$ \cite[Lemma 2.4.5]{HOS}. For each $a\in M$ we shall denote by $T_a$ the Jordan multiplication operator by the element $a$, that is, $T_a(x) = a\circ x$ ($x\in M$).\smallskip

An element $a$ in a unital Jordan Banach algebra $M$ is called \emph{invertible} whenever there exists $b\in M$ satisfying $a \circ b = \11$ and $a^2 \circ b = a.$ The element $b$ is unique and it will be denoted by $a^{-1}$ (cf. \cite[3.2.9]{HOS} and \cite[Definition 4.1.2]{Cabrera-Rodriguez-vol1}). We know from \cite[Theorem 4.1.3]{Cabrera-Rodriguez-vol1} that an element $a\in M$ is invertible if and only if $U_a$ is a bijective mapping, and in such a case $U_a^{-1} = U_{a^{-1}}$. \smallskip

As in the associative case, an \emph{involution} on a Jordan algebra $M$ is a mapping $a\mapsto a^*$ satisfying $(a^*)^* =a$ and $(a\circ b)^* = a^* \circ b^*$ for all $a,b\in M$. The involution $^*$ is called \emph{proper} if $a\circ a^* = 0 $ implies $a=0$. \smallskip

A very special source of examples is provided by associative algebras. Namely, suppose $A$ is a real or complex associative algebra with product denoted by juxtaposition. Then the natural Jordan product $a\circ b :=\frac12 ( a  b + b a)$ defines an structure of Jordan algebra on $A$; Jordan algebras of this type are called \emph{special}, as they are isomorphic to subalgebras of associative algebras equipped with a new multiplication (a term coined by Jordan, von Neumann \& Wigner \cite{JordanvonNeumannWigner34}). There are Jordan algebras which are not special (cf. \cite[Corollary 2.8.5]{HOS}), these algebras are called \emph{exceptional}.\smallskip

Suppose that $A$ is a C$^*$-algebra. The (associative) product of two self-adjoint elements in $A$ need not be, in general, self-adjoint. Another good property of the natural Jordan product assures that the Jordan product of two self-adjoint elements in $A$ also is in $A_{sa}$. Therefore, $A_{sa}$ is a real Jordan subalgebra of $A$, but not an associative subalgebra.\smallskip

A central notion in the study of Jordan algebras is the so-called $U$-mapping. Let $a,b$ be two elements in a Jordan algebra $M$. The $U_{a,b}$ mapping is the linear map on $M$ given by $$U_{a,b} (x) =(a\circ x) \circ b + (b\circ x)\circ a - (a\circ b)\circ x,$$ for all $x\in M$. The mapping $U_{a,a}$ is usually denoted by $U_a$. The $U$-maps satisfy the following fundamental identity: \begin{equation}\label{eq fundamental identity UaUbUa} U_a U_b U_a = U_{U_a(b)}, \hbox{ for all } a,b \hbox{ in a Jordan algebra } M \end{equation} (see \cite[2.4.18]{HOS}).\smallskip 

It is now the moment to introduce some analytic structures. A Jordan algebra $M$ endowed with a complete norm satisfying $\|a\circ b\| \leq \|a\| \|b\|,$ $a,b\in M$ is called a \emph{Jordan Banach algebra}. A \emph{JB-algebra} is a real Jordan Banach algebra $N$ whose norm  satisfies the following two geometric axioms:\begin{enumerate}[$(i)$]
	\item  $\|a^2\| = \|a\|^2$, for all $a\in N$;
	\item  $\|a^2\| \leq \| a^2+b^2\|$, for all $a,b\in N$ 
\end{enumerate} (see \cite[Definition 3.1.4]{HOS}).\smallskip

The Jordan mathematical model closest to C$^*$-algebras is given by JB$^*$-algebras. A \emph{JB$^*$-algebra} is a complex Jordan Banach algebra $M$ together with an algebra involution $a\mapsto a^*$, whose norm satisfies the following generalization of the Gelfand-Naimark axiom: $$\| U_a(a^*)\|=\|a\|^3, \hbox{ for every } a\in M.$$ 

Both of the just introduced Jordan structures are intrinsically related thanks to a result due to  J. D. M. Wright proving that every JB-algebra corresponds to the self-adjoint part of a (unique) JB$^*$-algebra  (see \cite{Wright77}).\smallskip 

If a C$^*$-algebra $A$ is equipped with its original norm and involution and the Jordan product given by $a\circ b=\frac{1}{2}(ab+ba)$, then the resulting structure is a JB$^*$-algebra. Jordan $^*$-subalgebras of C$^*$-algebras are called JC$^*$-algebras, and their symmetric parts are known as JC-algebras. The class of JB$^*$-algebras is strictly bigger than the collection of all associative C$^*$-algebras since, for example, the exceptional Jordan algebra $H_3(\mathbb{O})$ is a purely exceptional JB-algebra, that is, there is 
no nonzero homomorphism from $H_3(\mathbb{O})$ into a JC-algebra (cf. \cite[\S 7.2]{HOS}). \smallskip

From a purely algebraic point of view, a \emph{complex Jordan triple system} is a complex linear space $E$ equipped with a triple product $\{x,y,z\}$ which is bilinear and symmetric in $x,z$ and conjugate linear in the $y$ and satisfies the following ternary Jordan identity:
\begin{equation}\label{eq Jordan equation} L(x,y)\{a,b,c\}=\{L(x,y)a,b,c\}-\{a,L(y,x)b,c\}+\{a,b,L(x,y)c\},
\end{equation} for all $x,y,a,b,c\in E$, where $L(x,y):E\rightarrow E$ is the linear
mapping given by $L(x,y)z=\{x,y,z\}$.\smallskip

The analytic structures known as \emph{JB$^{*}$-triples}, whose origins go back the theory of holomorphic functions on infinite dimensional complex Banach spaces \cite{Ka}, are defined as those complex Jordan triple systems $E$ which are Banach spaces satisfying the next ``\emph{geometric}'' axioms:\begin{enumerate}[$(a)$]\item  For each $x\in E$, the operator $L(x,x)$ is hermitian with non-negative spectrum;
\item $\left\Vert \{x,x,x\}\right\Vert =\left\Vert x\right\Vert ^{3}$ for all $%
x\in E$.
\end{enumerate}\smallskip

The triple product of each JB$^*$-triple $E$ is a non-expansive mapping, that is, \begin{equation}\label{eq triple product is non-expansive} \|\{a,b,c\}\| \leq \|a\| \ \|b\|\ \|c\| \ \ \ (a,b,c\in E) \ \ \ \hbox{  \cite[Corollary 3]{FriRu86}.}\end{equation}

JBW$^*$-triples (respectively, JBW$^*$-algebras) are defined as those JB$^*$-triples (respectively, JB$^*$-algebras) which are also dual Banach spaces. The bidual of every JB$^*$-triple is a JBW$^*$-triple (see \cite{Di86}). It is further known that each JBW$^*$-triple admits a unique (isometric) predual and its product is separately weak$^*$ continuous \cite{BaTi} (see also \cite[Theorems 5.7.20 and 5.7.38]{Cabrera-Rodriguez-vol2}).\smallskip

Each C$^*$-algebra $A$ carries a natural structure of JB$^*$-triple with respect to the triple product given by \begin{equation}\label{eq ternary product on C*-algebras} \{a,b,c\} = \frac12 (a b^* c+ c b^* a).
\end{equation} The same triple product equips the space $B(H,K),$ of all bounded linear operators between two complex Hilbert spaces, with structure of JB$^*$-triple. In particular, we find infinite-dimensional complex Hilbert spaces which are JB$^*$-triples.\smallskip

For each element $a$ in a JB$^{*}$-triple $E$, the symbol $Q(a)$ will denote the conjugate linear operator on $E$ defined by $Q(a)(x)=\{a,x,a\}$. Every JB$^*$-algebra $M$ is a JB$^*$-triple with triple product \begin{equation}\label{eq triple prduct JBstar} \{a,b,c\} = (a\circ b^*) \circ c + (a\circ b^*) \circ c - (a\circ c) \circ b^*.
\end{equation} It follows that $Q(a) (x) = U_a (x^*)$ for all $a,x\in M$.\smallskip

We refer to \cite{HOS,Cabrera-Rodriguez-vol1} and \cite{Cabrera-Rodriguez-vol2} for the basic background on JB$^*$-triples and JB$^*$-algebras.\smallskip

An element $e\in E$ is called a \textit{tripotent} if $\{e,e,e\}=e$. When a C$^*$-algebra $A$ is regarded as a JB$^*$-triple with the triple product in \eqref{eq ternary product on C*-algebras}, it is known that the tripotents in $A$ are precisely the partial isometries in $A.$  In the same way that each partial isometry in a C$^*$-algebra $A$ induces a Peirce decomposition,  each tripotent $e$ in a JB$^*$-triple $E$ produces a \textit{ Peirce decomposition} of $E$ in the form $E=E_{2}(e)\oplus
E_{1}(e)\oplus E_{0}(e)$, where $E_{i}(e)$ is the $\frac{i}{2}$ eigenspace of the operator $L(e,e)$, $i=0,1,2$. This decomposition satisfies the following \emph{Peirce rules:} $$\{E_{2}(e),E_{0}(e),E\}=\{%
E_{0}(e),E_{2}(e),E\}=0$$ and $$\{E_{i}(e),E_{j}(e),E_{k}(e)\}\subseteq E_{i-j+k}(e),$$ when $i-j+k\in \{0,1,2\}$ and is zero otherwise. The \textit{Peirce} $k$\textit{-projection}, $P_{k}(e)$, is the natural projection of $E$ onto $E_{k}(e)$. Peirce projections are non-expansive (cf. \cite[Corollary 1.2]{FriRu85}) and they can be expressed in the following terms: $$P_{2}(e)=Q(e)^{2},\ P_{1}(e)=2(L(e,e)-Q(e)^{2}),$$ and  $$ P_{0}(e)=Id_{E}-2L(e,e)+Q(e)^{2} .$$

It is known that the Peirce-2 subspace $E_{2}(e)$ is a JB$^{*}$-algebra with unit $e$, Jordan product $x\circ _{e}y:=\{x,e,y\}$ and involution $x^{\ast _{e}}:=\{e,x,e\}$, respectively. It is worth to note that a linear bijection between JB$^*$-triples is an isometry if and only if it is a triple isomorphisms (cf. \cite[Proposition 5.5]{Ka}). Consequently, the triple product in $E_{2}(e)$ is uniquely given by $$\{x,y,z\}=(x\circ _{e}y^{\ast _{e}})\circ _{e}z+(z\circ _{e}y^{\ast
_{e}})\circ _{e}x-(x\circ _{e}z)\circ _{e}y^{\ast _{e}},$$ $x,y,z\in E_{2}(e)$.
\smallskip

A subspace $B$ of a JB$^*$-triple $E$ is a JB$^*$-subtriple of $E$ if $\{B,B,B\}\subseteq B$. A JB$^*$-subtriple $I$ of $E$ is called an \emph{inner ideal} of $E$ if $\{I,E,I\}\subseteq I$. A subspace $I$ of a C$^*$-algebra $A$ is an \emph{inner ideal} if $I A I \subseteq I$. Every hereditary $\sigma$-unital C$^*$-subalgebra of a C$^*$-algebra is an inner ideal. A complete study on inner ideals of JB$^*$-triples is available in \cite{EdRu92} and the references therein. It follows from Peirce rules that for each tripotent $e$ in a JB$^*$-triple $E$, the Peirce-2 subspace $E_2(e)$ is an inner ideal.\smallskip


Let $E$ be a JB$^*$-triple. The JB$^*$-subtriple, $E_a$, of $E$ generated by a single element $a$ is identified, via the Gelfand theory, with the commutative C$^*$-algebra $$C_0(\Omega_{a})=\{f: \Omega_a\to \mathbb{C} \hbox{ continuous with } f(0) =0 \hbox{  if  } 0\in \Omega_a\},$$  for a unique compact set $\Omega_{a}$ contained in $[0,\|a\|],$ such that $\|a\|\in \Omega_{a}$ and $0$ cannot be isolated in $\Omega_a$; and under this identification $a$ corresponds to the continuous function given by the embedding of $\Omega_{a}$ into $\mathbb{C}$ (cf. \cite[Corollary 1.15]{Ka} and \cite[Lemma 3.2]{Ka96}). A consequence of this representation affirms that every element in a JB$^*$-triple admits a cubic root and a $(2n-1)$th-root ($n\in \mathbb{N}$) belonging to the JB$^*$-subtriple that it generates. The sequence $(a^{[\frac{1}{2n-1}]})$ of all  $(2n-1)$th-roots of $a$ converges in the weak$^*$ 
(and also in the strong$^*$) topology of $E^{**}$ to a tripotent in $E^{**},$ denoted by $r_{{E^{**}}}(a)$, and called the \emph{range tripotent} of $a$. The tripotent $r_{{E^{**}}}(a)$
is the smallest tripotent $e\in E^{**}$ satisfying that $a$ is positive in the JBW$^*$-algebra $E^{**}_{2} (e)$. It is also known that, if $\|a\|= 1,$ the sequence $(a^{[2n -1]}),$ of all odd-powers of $a$, converges in the weak$^*$- and strong$^*$-topology of $E^{**}$ to a tripotent (called the
\hyphenation{support}\emph{support} \emph{tripotent} of $a$, $u(a)$ in $E^{**}$, which satisfies $ u(a) \leq a \leq r_{{E^{**}}}(a)$ in $E^{**}_2 (r_{{E^{**}}}(a))$ (compare \cite[Lemma 3.3]{EdRu88}; beware that in \cite{EdRu96}, $r(x)$ is called the support tripotent of
$x$). In case that $a$ is a positive element in a JB$^*$-algebra $M$, the support and the range tripotents of $a$ in $M^{**}$ are projections, called the \emph{support} and \emph{range projections} of $a$ in $M^{**}.$ \smallskip

For each element $a$ in a JB$^*$-triple $E$ (in which we generally do not have a cone of positive elements), the symbol $E(a)$ will stand for the norm-closure of $\J aEa = Q(a) (E)$ in $E$. It was proved by L.J. Bunce, C.-H. Chu and B. Zalar that $E(a)$ is precisely the norm-closed inner ideal of $E$ generated by
$a$. Clearly, $E_a\subset E(a)$.  It is further shown in the just quoted reference that $E(a)$ is a JB$^*$-subalgebra of the JBW$^*$-algebra $E(a)^{**} = \overline{E(a)}^{w^*} =
E^{**}_{2} (r_{{E^{**}}}(a))$ and contains $a$ as a positive element, where $r_{{E^{**}}}(a)$  is the range tripotent of $x$ in
$E^{**}$ (cf. \cite[Proposition 2.1]{BunChuZal2000}).\label{inner ideal}\smallskip

The reader will need some basic knowledge on the strong$^*$-topology of a JB$^*$-triple. If we are given a norm-one functional $\varphi$ in the predual, $W_*$, of a JBW$^*$-triple $W$, and a norm-one element $z$ in $W$ with $\varphi (z) =1$, the mapping 
$$(x,y)\mapsto \varphi\J xyz$$ defines a positive sesquilinear
form on $W$. Moreover, the mapping does not depend on the chosen $z,$ that is, if $w\in W$ satisfies $\varphi (w) =1$, we have $\varphi\J xyz = \varphi\J
xyw,$ for all $x,y\in W$  (see \cite[Proposition 1.2]{barton1987grothendieck}). The mapping $ x\mapsto \|x\|_{\varphi}:=
\left(\varphi\J xxz\right)^{\frac{1}{2}},$ defines a prehilbertian seminorm on $W$. The strong$^*$-topology (denoted by $S^*(W,W_*)$) is the topology on $W$ generated by the family of all semi-norms $\|\cdot\|_{\varphi}$ with $\varphi$ running in the unit sphere of the predual of $W$ (cf. \cite{BarFri90}). For the purposes of this note we recall that the triple product of every JBW$^*$-triple $W$ is jointly strong$^*$ continuous on bounded sets. The first proof of this result appeared in \cite{RodPa91}, however the difficulties affecting Grothendieck's inequalities in \cite{barton1987grothendieck} also impacted the original proof and an alternative argument can be found in \cite[Theorem 9]{PeRo2001}. The recent proof of the Barton-Friedman conjecture on Grothendieck's inequalities for JB$^*$-triples in \cite{HamKaPePfi-BFc} reinstates the validity of the original proof.\smallskip

The strong$^*$-topology of a JB$^*$-triple $E$ is defined as the restriction to $E$ of the strong$^*$-topology of its bidual. \smallskip

The notion of orthogonality for non-necessarily hermitian elements in a JB$^*$-algebra actually requires to identify JB$^*$-algebras inside the class of JB$^*$-triples. The general notion reads as follows: elements $a,b$ in a JB$^*$-triple $E$ are
said to be \emph{orthogonal} (written $a\perp b$) if $L(a,b) =0$. It is known that $a\perp b$ if and only if $b\perp a$ if and only if $E(a)\perp E(b)$ (see \cite[Lemma 1]{BurFerGarMarPe2008} for additional details).\smallskip

\section{An orderless approach to Rickart C*-algebras}\label{sec: Rickart Cstar-algebras} 

This section is devoted to explore some equivalent reformulations of the notions of (weakly) Rickart and Baer C$^*$-algebras in which we do not need the natural partial order nor the cone of positive elements. Our departure point is a result by G.K. Pedersen from \cite{PedSAW} where a reformulation in terms of hereditary subalgebras is established.\smallskip

We begin by recalling the definition of another class of C$^*$-algebras introduced by G.K. Pedersen in \cite{PedSAW}. A \emph{SAW$^*$-algebra} is a C$^*$-algebra $A$ satisfying that for any two orthogonal positive elements $x$ and $y$ in $A$ there is a positive element $e$ in $A$ (which is not assumed to be a projection) such that $e x = x$ and $ey = 0$. In the commutative setting these SAW$^*$-algebras correspond to C$^*$-algebras of the form $C_0(L)$ for some sub-Stonean (locally compact Hausdorff) space $L$. It should be remarked that sub-Stonean spaces, studied by K. Grove and G.K. Pedersen in \cite{GroPed84}, are defined as those  locally compact Hausdorff spaces in which disjoint $\sigma$-compact open subspaces have disjoint compact closures.  \smallskip

A C$^*$-subalgebra $B$ of a C$^*$-algebra $A$ is said to be a \emph{hereditary C$^*$-subalgebra} of $A$ if whenever $0 \leq a\leq  b$ with $a \in  A$ and $b \in B$, then $a\in B$, equivalently, $B^+$ is a face of $A^+$. It is known that an hereditary C$^*$-subalgebra $B$ of C$^*$-algebra $A$ is $\sigma$-unital if and only if it has the form $B= \overline{xAx}$ for some positive $x\in A$.\smallskip

\begin{proposition}\label{Prop Pedersen}\label{page Pedersen Characterization}\cite[Proposition 1]{PedSAW} Let $A$ be a C$^*$-algebra. Consider the following condition: Given two orthogonal, hereditary C$^*$-subalgebras $B$ and $C$ of $A$, there is an element $e$ in $A^+$ which is a unit for $B$ and annihilates $C$. Then the following statements hold: \begin{enumerate}\item[(AW$^*$)] If this condition holds for all pairs $B$, $C$, then $A$ is an AW$^*$-algebra; 
\item[(WRC$^*$)] If this condition holds when $B$ is $\sigma$-unital and $C$ is arbitrary, then $A$ is a weakly Rickart C$^*$-algebra;
\item[(SAW$^*$)] If this condition is true when both $B$ and $C$ are $\sigma$-unital, then $A$ is a SAW$^*$-algebra.
\end{enumerate}
\end{proposition}

\begin{remark}\label{r Pedersen is a characterization}{\rm
It should be noted that the implications in (AW$^*$), (WRC$^*$) and (SAW$^*$) are actually equivalences and characterizations of AW$^*$-algebras, weakly Rickart C$^*$-algebra, and SAW$^*$-algebras. Namely, if $A$ is an AW$^*$-algebra and $B$ and $C$ are two orthogonal, hereditary C$^*$-subalgebras of $A$, by the hypothesis on $A$, there exists a projection $p$ in $A$ such that $R(C) = p A.$ Clearly, $(\mathbf{1}-p) c= c,$ for all $c\in C,$ and since $B$ and $C$ are orthogonal, $B\subset R(C) = p A$. Since $B$ and $C$ are self-adjoint, $p$ is a unit for $B$ and annihilates $C$. If $A$ is a weakly Rickart C$^*$-algebra, $B$ is the closure of $x A x$ for some positive $x$, and $C$ is arbitrary, by the assumptions on $A$, there exists a projection $p$ in $A$ such that $x p = x$ and $x y =0$ implies $p y =0$. Therefore $p$ is a unit for $B$ and annihilates $C$. The remaining equivalence can be similarly obtained.\smallskip

Although it is not explicit in \cite[Proposition 1]{PedSAW}, the following equivalence also holds by the same arguments:
\begin{enumerate} \label{equivalence RCstar}
\item[(RC$^*$)] The condition in Proposition \ref{Prop Pedersen} holds when $C$ is $\sigma$-unital and $B$ is arbitrary if, and only if, $A$ is a Rickart C$^*$-algebra.
\end{enumerate}
}\end{remark}

Let us briefly recall some basic facts on range projections. Suppose $a$ is a positive element in a von Neumann algebra $W$. The \emph{range projection} of $a$ in $W$ (denoted by $rp(a)$) is the smallest projection $p$ in $W$ satisfying $a p = a$. It is known that the sequence $\left( (1/n \mathbf{1}+a)^{-1} a \right)_n$ is monotone increasing to $rp(a)$, and hence it converges to $rp(a)$ in the weak$^*$-topology of $W$. If $a$ is in the closed unit ball of $A$, the sequence $(a^{\frac1n})_n$ is monotone increasing and converges to $rp(a)$ in the weak$^*$-topology of $A$. Actually, for any element  $x$ in $W$, the smallest projection $l$ in $W$ with $l x = x$ is called the \emph{left range projection} of $x$ and denoted by $s_l(x)$. The \emph{right range projection} $s_r(x)$ is the smallest projection $q$ in $W$ with $x q = x$ (cf. \cite[Definition 1.4]{Tak} or \cite[2.2.7]{Ped}).  It is known that $r(x^* x ) = s_r(x)$ and $r(x x^*) = s_l(x)$, while $r(x x^*) = s_l(x)= s_r(x)$ for any self-adjoint $x$. If $x$ is an element in a C$^*$-algebra $A$, we shall usually employ the left and right range projections of $x$ in $A^{**}$. If $A$ is a Rickart C$^*$-algebra, for each $x\in A$ we have $s_r(x)\leq RP(x)$ and $s_l(x)\leq LP(x)$ in $A^{**}$.\smallskip

An element $e$ in a C$^*$-algebra $A$ is a partial isometry if $ee^*$ (equivalently, $e^*e$) is a projection in $A$. Each partial isometry $e\in A$ induces a Peirce decomposition of $A$ in the form $A = A_2(e)\oplus A_1(e) \oplus A_0(e)$, where $A_2 (e) = ee^* A e^*e $, $A_1 (e) = (\mathbf{1}-ee^*) A e^*e \oplus ee^* A (\mathbf{1}-e^*e),$ and $A_0(e) =  (\mathbf{1}-ee^*) A (\mathbf{1}-e^* e).$ The subspace $A_j(e)$ is called the Peirce-$j$ subspace. The Peirce-$2$ subspace $A_2(e)$ is a unital C$^*$-algebra, with unit $e,$ when equipped with the original norm, product $a\bullet_e b = a e^* b$ and involution $a^{*_e} = e a^* e$ ($a,b\in A$). \smallskip

A couple of projections $p,q$ in a C$^*$-algebra $A$ are said to be  (Murray-von Neumann) \emph{equivalent}, $p\sim q$, if $p = ee^*$ and $q= e^*e$ for some partial isometry $e \in A$. A unital C$^*$-algebra $A$ is finite if $p\sim \11$ implies $p = \11.$\smallskip

In our seeking of an order-free characterization of (weakly) Rickart C$^*$-algebras, which can be employed to define an appropriate notion in general JB$^*$-triples, we shall need the following milestone result due to P. Ara: \emph{``Left and right projections are (Murray-von Neumann) equivalent in Rickart C$^*$-algebras''} (see \cite{Ara89} where this famous conjecture by I. Kaplansky was proved). The same conclusion actually holds for weakly Rickart C$^*$-algebras. The result is included here for the lacking of an explicit reference. 

\begin{lemma}\label{l Ara's theorem for weakly Rickark} Let $A$ be a weakly Rickart C$^*$-algebra. Then the left and right projections of every element in $A$ are Murray-von Neumann equivalent.
\end{lemma}

\begin{proof} Let $x$ be an element in a weakly Rickart C$^*$-algebra $A$. If $A$ is unital, the conclusion follows from Ara's theorem \cite[Theorem 2.5]{Ara89}. So, we shall assume that $A$ is non-unital.\smallskip

By \cite[Theorem 5.1]{BerBaerStarRings} (see also \cite[Lemma 3.6]{SaitoWright2015}) we can find a unitization $A_{\mathbf{1}}= A\oplus \mathbb{C} \mathbf{1}$ of $A$ which is a Rickart C$^*$-algebra. Fix $x\in A$. Let $RP (x)$ and $LP (x)$ denote the right and left projections of $x$ (in $A$ or in $A_{\mathbf{1}}$).  Let us observe that these symbols offer no ambiguity. More concretely, if $e= LP_A (x) $ is the ALP of $x$ in $A$, Lemma 5.3 in \cite{BerBaerStarRings} assures that $e$ is the ALP of $x\in A$ in $A_{\mathbf{1}}$, that is, $LP_{A}(x)= LP_{A_{\mathbf{1}}}(x)$.  Similarly, $RP_{A}(x)= RP_{A_{\mathbf{1}}}(x)\in A$.\smallskip

By applying \cite[Theorem 2.5]{Ara89} we deduce that $LP(x)$ and $RP (x)$ are equivalent projections in $A_{\mathbf{1}}$, that is, there exists a partial isometry $e\in A_{\mathbf{1}}$ such that $e^* e = RP(x)$ and $e e^* =LP(x)$.\smallskip

We shall finally show that $e\in A$. Let us write $e= e_1 +\lambda \mathbf{1}$ with $e_1$ in $A$ and $\lambda\in \mathbb{C}$. Since $A\ni LP (x) = e e^* = e_1 e_1^* + \lambda e_1^* + \overline{\lambda} e_1 + |\lambda|^2 \mathbf{1},$ it follows that $\lambda=0,$ and thus $e = e_1\in A$.
\end{proof}

\begin{remark}\label{r wR Cstar algebras are generated by their projections}{\rm We have already commented that the idea behind Rickart's original paper \cite[Theorem 2.10]{Rick46} (see also \cite[Proposition 8.1]{BerBaerStarRings}) was to show that every Rickart C$^*$-algebra is generated by its projections. Actually, the same occurs for weakly Rickart C$^*$-algebras. Namely, let $a$ be a positive element in a weakly Rickart C$^*$-algebra $A$. Let $p= RP(a)$ denote the ARP projection of $a$ in $A$ when the latter is regarded as a weakly Rickart C$^*$-algebra. It follows from \cite[Proposition 5.6]{BerBaerStarRings} that $p A p$ is a Rickart C$^*$-algebra with unambiguous left and right projections for every element in $p A p.$ It follows from the mentioned Theorem 2.10 in \cite{Rick46} that $p A p$ is generated by its projections. In particular $a\in pAp$ can be approximated in norm by finite linear combinations of projections. 
}\end{remark}


Given a positive element $a$ in a C$^*$-algebra $A$, the
hereditary C$^*$-subalgebra of $A$ generated by $a$ coincides with the norm closure, $ \overline{a A a}$, of $a A a$ and contains the C$^*$-subalgebra generated by $a$ (see \cite[Corollary 3.2.4]{Murph}). This hereditary C$^*$-subalgebra is precisely the inner ideal of $A$ generated by $a$, when $A$ is regarded as a JB$^*$-triple (cf. \cite[pages 19-20]{BunChuZal2000}). Therefore the symbol $A(a)$  will denote the hereditary C$^*$-subalgebra of $A$ generated by $a$. It is further known, even in a more general setting, that $A(a)^{**}$ identifies with $(A^{**})_2 (rp(a)) = rp(a) A^{**} rp(a)$ (because $rp(a)$ is a projection), and $A(a)$ is actually a C$^*$-subalgebra of this latter hereditary C$^*$-subalgebra of $A^{**}$ (cf. \cite[Proposition 2.1.]{BunChuZal2000} whose proof is valid here too). It is worth mentioning that \begin{equation}\label{eq inner ideal intersection with A positive} A(a) = (A^{**})_2 (rp(a))\cap A.
\end{equation} Namely, the inclusion $\subseteq$ is clear. We may clearly assume that $\|a\|\leq 1$. On the other hand, it is not hard to see that $a$ is a strictly positive element in the hereditary C$^*$-subalgebra $I =(A^{**})_2 (rp(a))\cap A$, and hence $(a^{\frac1n})_n$ is an approximate identity in $I$ (cf. \cite[Exercise 3 in page 31]{Tak}). Given $x\in I$, the sequence  $(a^{\frac1n} x a^{\frac1n})_n$ converges in norm to $x$ and is contained in $A(a)$, therefore $x\in A(a).$\smallskip

It is well known that every $\sigma$-unital hereditary C$^*$-subalgebra of $A$ is of the form $A(x)$, with $x$ positive in $A$ (cf. \cite[Theorem 3.2.5.]{Murph}, see also \cite[\S 1.5]{Ped} and \cite[\S 3.2]{Murph} for a detailed discussion on hereditary C$^*$-subalgebras and ideals). Moreover, as commented by G.K. Pedersen in \cite[page 16]{PedSAW}, $\sigma$-unital hereditary C$^*$-subalgebras of $A$ can be also represented in the form $\overline{(A y)}\cap \overline{(y^* A)}$ with $y\in A$. Clearly, each hereditary C$^*$-subalgebra of the form $A(a)$ with $a\geq 0$ writes in the form  $\overline{(A a)}\cap \overline{(a^* A)}$ (just apply \eqref{eq inner ideal intersection with A positive} in the non-trivial inclusion). On the other direction, for each $y\in A$, we shall show that $\overline{(A y)}\cap \overline{(y^* A)} = A(y^* y).$ Indeed, since $\overline{(A y)}\cap \overline{(y^* A)}$ is an inner ideal and contains $y^* y$, we deduce that $\overline{(A y)}\cap \overline{(y^* A)}\supseteq A(y^*y).$  If we take $z\in \overline{(A y)}\cap \overline{(y^* A)}$, we clearly have $r(y^* y ) z = z r(y^* y)$, and thus $\overline{(A y)}\cap \overline{(y^* A)}\subseteq A(y^* y)$ (cf. \eqref{eq inner ideal intersection with A positive}).\smallskip

For a general element $x$ in a C$^*$-algebra $A$, the  inner ideal of $A$ generated by $x$ can be described as the norm closure of $x A x$ (cf. \cite[pages 19-20]{BunChuZal2000}).\smallskip

Let us recall that elements $a,b$ in a C$^*$-algebra $A$ are called orthogonal ($a\perp b$ in short) if $a b^* = b^* a =0.$ The orthogonal complement of a subset $\mathcal{S}\subset A$ is defined as $\mathcal{S}^{\perp} :=\{ a\in A : a\perp x \hbox{ for all } x\in\mathcal{S}\}.$

\begin{proposition}\label{p sufficient and necessary conditions in terms of inner ideals without order} Let $A$ be a C$^*$-algebra. Then the following statements hold:  \begin{enumerate}[$(a)$]\item $A$ is a weakly Rickart C$^*$-algebra if, and only if, given $x\in A$ and an inner ideal $J\subseteq A$ with $I = A(x)\perp J $, there exists a partial isometry $e$ in $A$ such that $I\subseteq A_2(e)$ and $J\subseteq A_0(e)$;
\item $A$ is a Rickart C$^*$-algebra if, and only if, $A$ is unital and given $x\in A$ and an inner ideal $J\subseteq A$ with $I = A(x)\perp J $, there exists a partial isometry $e$ in $A$ such that $I\subseteq A_2(e)$ and $J\subseteq A_0(e)$.    
\end{enumerate} 
\end{proposition}

\begin{proof}$(a)$ ($\Rightarrow$) By Lemma \ref{l Ara's theorem for weakly Rickark} $LP(x)$ and $RP (x)$ are equivalent projections in $A$, that is, there exists a partial isometry $e\in A$ such that $e^* e = RP(x)$ and $e e^* =LP(x)$. Clearly, $x\in A_2(e),$ and hence $\{x,A,x\}\subseteq A_2(e)$. It follows that $A(x) \subseteq A_2(e)$. \smallskip

On the other hand, 	for each $y\in J \perp A(x)$ we have $ x^* y =  0 = y x^*,$ and since $ e^*e = RP (x) = LP (x^*)$ and $e e^* = LP (x) = RP (x^*)$ we deduce that $ e e^* y = 0 = y e^* e$, witnessing that $y \in A_0 (e)$.\smallskip
	
($\Leftarrow$) This implication follows from Proposition \ref{Prop Pedersen} and its proof in \cite[Proposition 1]{PedSAW}, we shall revisit the argument for completeness. Fix $y\in A$ and consider the inner ideal $I= \overline{(A y)} \cap \overline{(y^* A)} = A(y^* y)$. Let $R= R(y)$ denote the right annihilator of $y$ in $A$. In this case $R\cap R^* = \{y^* y\}^{\perp}:=J$. By the assumptions, there exists a partial isometry $e\in A$ such that $$I\subseteq  A_2(e) = ee^* A e^*e \hbox{ and } J \subseteq A_0(e)  = (\mathbf{1}-ee^*) A (\mathbf{1}-e^*e).$$  Therefore, $ y^*y = ee^* y^*y e^*e $, and thus $s_{r}(y) e^*e= r_{_{A^{**}}} (y^*y) e^*e = r_{_{A^{**}}} (y^*y)= s_{r}(y)$ in ${A^{**}}$ and $y e^* e= y r_{_{A^{**}}} (y^*y) e^* e = y r_{_{A^{**}}} (y^*y) = y$. \smallskip

If $z\in R$, we have $zz^* \in  R\cap R^* \subseteq A_0(e)  = (\mathbf{1}-ee^*) A (\mathbf{1}-e^*e)$, which proves that $zz^* = (\mathbf{1}-ee^*) zz^* (\mathbf{1}-e^*e),$ and $zz^* = (\mathbf{1}-e^*e) zz^* (\mathbf{1}-ee^*).$ By repeating the arguments above we get $ (\mathbf{1} - e^* e) r_{_{A^{**}}} ( zz^*) = r_{_{A^{**}}} ( zz^*)$ leading to $e^* e s_l(z)= 0$ in $A^{**}$, and to $e^* e z = e^* e s_l(z) z = 0$. \smallskip

$(b)$ This is clear from $(a)$ and the fact that a C$^*$-algebra is a Rickart C$^*$-algebra if and only if it is weakly Rickart and unital (cf. \cite[Proposition 5.2]{BerBaerStarRings}).
\end{proof}

The advantage of the previous proposition is that the equivalent reformulations do not depend on the natural partial order given by the cone of positive elements in a C$^*$-algebra. \smallskip

\begin{remark}\label{r SAWstar algebras only sufficient} Let $A$ be a C$^*$-algebra. Clearly $A$ is a SAW$^*$-algebra if given $x,y\in A$ with $x\perp y$, there exists a partial isometry $e$ in $A$ such that $I = A(x)\subseteq A_2(e)$ and $J = A(y)\subseteq A_0(e)$ {\rm(}cf. \cite[Proposition 1]{PedSAW}{\rm)}. We do not know if the reciprocal implication holds. The lacking of an analogue of Ara's theorem in \cite[Theorem 2.5]{Ara89} proving the equivalence of left and right projections in the setting of SAW$^*$-algebras seems to be a major obstacle. 	
\end{remark}

In the light of Pedersen's result in Proposition \ref{Prop Pedersen} and the characterization in terms of inner ideals given in Proposition \ref{p sufficient and necessary conditions in terms of inner ideals without order}, it seems natural to ask if a characterization of Baer or AW$^*$-algebras can be obtained in terms of inner ideals. If we assume some extra hypothesis the answer is yes.

\begin{proposition}\label{p characterization of finite Rickart and Baer Cstar algebras by inner ideals} Let $A$ be a finite unital C$^*$-algebra. Then the following statements hold:  \begin{enumerate}[$(a)$]\item $A$ is a Rickart C$^*$-algebra if, and only if, given $x\in A$ and an inner ideal $J\subseteq A$ with $I = A(x)\perp J $, there exists a partial isometry $e$ in $A$ such that $J\subseteq A_2(e)$ and $I\subseteq A_0(e)$;
\item  $A$ is an AW$^*$-algebra if, and only if, for any family $\{x_i\}_{i}$ of mutually orthogonal elements in $A$ and each inner ideal $J\subseteq A$ with $A(x_i)\perp J$ for all $i$, there exists a partial isometry $e\in A$ satisfying $J\subseteq A_2(e)$ and $A(x_i)\subseteq A_0(e)$ for all $i$.
\end{enumerate} 
\end{proposition}

\begin{proof}$(a)$ $(\Rightarrow)$ Let us fix $x\in A$. Since $A$ is a finite Rickart C$^*$-algebra, $LP(x)$ and $RP(x)$ are unitarily equivalent \cite[Theorem 4.1$(c)$]{Hand79}, that is, there exists a unitary $u\in A$ such that $RP(x) = u LP(x) u^*,$  and hence $\11-RP(x) = u (\11-LP(x)) u^*.$ Set $e=  (\11-LP(x)) u^*.$ Clearly, $e$ is a partial isometry with $ee^* = \11- LP(x)$ and $e^* e = \11- RP(x),$ and $x\in LP(x)\ A \  RP(x) = (\11-ee^*) A (\11-e^* e) = A_0 (e).$ This proves that $A(x) \subseteq A_0(e).$\smallskip

If we take $y \in J\perp I$, it follows that $y x^* = x^* y=0,$ which implies that $y RP(x) = LP(x) y=0,$ and thus $y\in (\11-LP(x)) A (\11-RP(x)) = A_2(e).$\smallskip

$(\Leftarrow)$ is a consequence of the equivalence in (RC$^*$) in page \pageref{equivalence RCstar}.\smallskip

$(b)$ $(\Rightarrow)$ Suppose $A$ is an AW$^*$-algebra (the projections in $A$ form a complete lattice \cite[Proposition 4.1]{BerBaerStarRings}). Let us take a family $\{x_i\}_{i}$ of mutually orthogonal elements in $A$ and an inner ideal $J\subseteq A$ with $A(x_i)\perp J$ for all $i$. It follows from the hypothesis that $RP(x_i)\perp RP(x_j)$ and $LP(x_i)\perp LP(x_j),$ for all $i\neq j$. By  \cite[Theorem 4.1$(c)$]{Hand79} $LP(x_i)$ and $RP(x_i)$ are unitarily equivalent, and hence equivalent via a partial isometry $w_i$ for all $i$. Theorem 20.1$(iii)$  in \cite{BerBaerStarRings} proves the existence of a partial isometry $w$ such that $ww^* = \bigvee_i LP(x_i),$ $w^* w = \bigvee_i RP(x_i)$ and $w RP(x_i) = w_i = LP(x_i) w$ for all $i$ (i.e. orthogonal partial isometries in an AW$^*$-algebra are addable). Applying once again that $A$ is a finite Rickart C$^*$-algebra we deduce that $w w^*$ and $w^* w$ are unitarily equivalent \cite[Theorem 4.1$(c)$]{Hand79}, and thus $\11 -w w^*$ and $\11-w^* w$ are equivalent. Let us take a partial isometry $e$ in $A$ with $e e^* = \11 -w w^*$ and $e^* e = \11-w^* w.$ It is easy to check that, by construction, $$\begin{aligned}
A_0(e) &= (\11-ee^*)A (\11-e^* e) = ww^* A w^* w \\
&= \left(\bigvee_i LP(x_i)\right) A \left(\bigvee_i RP(x_i) \right)\supset LP(x_{i_0}) A RP(x_{i_0}) =  A(x_{i_0}),
\end{aligned}
$$  for all $i_0$. Given $y \in J$ and an index $i_0$, it follows from the properties of the left and right projections of $x_{i_0}$ and the fact that $J\perp x_{i_0},$ that $$J \subseteq (\11- LP(x_{i_0})) A (\11- RP(x_{i_0})), \hbox{ for all } i_0,$$ and thus $$\begin{aligned} J &\subseteq \left( \bigwedge_i  (\11- LP(x_{i_0}))\right) A \left( \bigwedge_i  (\11- RP(x_{i_0}))\right) \\ &= \left( \11 - \bigvee_i  LP(x_{i_0}) \right) A \left(   \11- \bigvee_i RP(x_{i_0})\right) = ee^* A e^* e = A_2 (e).
\end{aligned}
$$

$(\Leftarrow)$ It follows from $(a)$ that $A$ is a Rickart C$^*$-algebra, and thus $A$ is unital. Let $\{p_i\}_{i\in \Gamma}$ be a family of mutually orthogonal projections in $A$. By applying the hypothesis to the inner ideal $J := \{x\in A : x\perp p_i \hbox{ for all } i\in \Gamma\}$, we deduce the existence of a partial isometry $e\in A$ such that $J\subseteq A_2(e)$ and $A(p_i) = A_2(p_i)\subseteq A_0(e)$ for all $i\in \Gamma$. The element $q = \11 - ee^*$ is a projection in $A$ satisfying $q p_i = p_i$ (equivalently, $q\geq p_i$) for all $i\in \Gamma.$ Let $r$ any other projection in $A$ with $r\geq p_i$ for all $i\in \Gamma.$ The property $(\11-r) p_i = 0$ for all $i$, implies that $\11-r\in J\subseteq A_2(e),$ and thus $ee^* (\11-r) = \11 -r$ and $(\11-ee^*) (\11-r) = 0,$ witnessing that $q= \11-ee^*\leq r,$ and therefore $q = \bigvee_i p_i$ in $A$. We have shown that every orthogonal family of projections in $A$ has a supremum. Proposition 4.1$(a)\Leftrightarrow(c)$ in  \cite{BerBaerStarRings} proves that that $A$ is an AW$^*$-algebra. Actually, by applying the same argument with $\11 - e^*e$ instead $\11 - ee^*$ we get $\11 - e^*e = \bigwedge_i p_i = \11 - ee^*$. 
\end{proof}

\begin{remark} The characterization provided in Proposition \ref{p characterization of finite Rickart and Baer Cstar algebras by inner ideals} is not valid without the hypothesis of finiteness. Consider, for example, the Hilbert space $H = \ell_2$ with orthonormal basis $\{\xi_n: n\in \mathbb{N}\}$ and $A=B(H)$. Take a partial isometry $v$ such that $\11-vv^* = \xi_1\otimes \xi_1$ is a rank-one projection and $\11-v^* v = \xi_1\otimes \xi_1 + \xi_2\otimes \xi_2$ has rank $2$. If for $J = \{v\}^{\perp}= A_0(v)\perp A(v) = A_2(v)$ there were a partial isometry $e$ satisfying that $A_0(v) =J \subseteq A_2(e)$ and $A_2(v) \subseteq A_0(e)$ we would have $\xi_i\otimes \xi_j \perp e$ for all $j\geq 3,$ $i\geq 2$ (because $\xi_i\otimes \xi_j\in A_2(v)$ for such $i$ and $j$). In particular, $e\perp \xi_2 \otimes \xi_3$, which gives $e^* (\xi_2 \otimes \xi_3) =0,$ and consequently  $ee^* (\xi_2) =0.$ On the other hand, $\xi_1\otimes \xi_1$ and $\xi_1\otimes \xi_2$ belong to $A_0(v)\subseteq A_2(e),$ and hence $\xi_1\otimes \xi_2 =\{e,e, \xi_1\otimes \xi_2\} = \frac12 (ee^* (\xi_1\otimes \xi_2) + (\xi_1\otimes \xi_2) e^* e)$ and $$2 \xi_1 = ee^* (\xi_1) + \langle e^*e (\xi_2),\xi_2\rangle \xi_1 = ee^* (\xi_1),$$ which is imposible since $\|ee^* (\xi_1)\| \leq 1.$
\end{remark}

\begin{remark}\label{r uniqueness of the partial isometry in Proposition charact inner ideals Cstar} The partial isometry $e$ appearing in the statements of Proposition \ref{p sufficient and necessary conditions in terms of inner ideals without order} need not be unique. Actually if $e$ is a partial isometry satisfying the desired conclusion, then the partial isometry $\lambda e$ satisfies the same property for all $\lambda$ in the unit sphere of $\mathbb{C}$.
\end{remark}

The partial isometry $e$ appearing in Proposition \ref{p sufficient and necessary conditions in terms of inner ideals without order}$(a)$ induces a local order in the C$^*$-algebra $(A_2(e),\bullet_e,*_e)$ and we actually obtain a strengthened version of the statement. 

\begin{proposition}\label{p sufficient and necessary conditions in terms of inner ideals without order but local order} Let $A$ be a C$^*$-algebra. Then $A$ is a weakly Rickart C$^*$-algebra if, and only if, given $x\in A$ and an inner ideal $J\subseteq A$ with $I = A(x)\perp J $, there exists a partial isometry $e$ in $A$ such that $I\subseteq A_2(e),$ $e^*e =RP(x)$, $e e^* = LP(x),$ $x$ is a positive element in the C$^*$-algebra $(A_2(e), \bullet_e, *_e),$ $A(x)$ is a C$^*$-subalgebra of the latter C$^*$-algebra and $J\subseteq A_0(e)$.
\end{proposition}

\begin{proof} It suffices to prove the extra properties in the ``only if'' implication. Suppose $A$ is a weakly Rickart C$^*$-algebra. We shall assume that $A$ is non-unital, and its unitization $A_{\mathbf{1}} = A\oplus \mathbb{C} \mathbf{1}$ is a Rickart C$^*$-algebra \cite[Theorem 5.1]{BerBaerStarRings} (see also \cite[Lemma 3.6]{SaitoWright2015}).\smallskip 

Fix $x\in A.$ Another essential contribution by P. Ara and D. Goldstein (see \cite[Corollary 3.5]{AraGold1993}, \cite[Corollary 7.4]{Gold95}) assures the existence of a polar decomposition for $x,$ that is, there exists a partial isometry $e\in A_{\mathbf{1}}$ such that $x = e |x|$, $ee^* = LP(x)$ and $e^*e = RP(x)$ (cf. also \cite[Proposition 21.3]{BerBaerStarRings}). If we write $e$ in the form $e= e_1 +\lambda \mathbf{1}$ with $\lambda\in \mathbb{C}$, $e_1\in A$. Since $e_1 e_1^* + \lambda e_1^* + \overline{\lambda} e_1 + |\lambda|^2 \mathbf{1}  = e e^* = LP (x)\in A,$ we infer that $e = e_1\in A$, that is weakly Rickart C$^*$-algebras satisfy polar decomposition.\smallskip

Let $I= A(x)$ and let $J$ be an inner ideal orthogonal to $I$. By considering the partial isometry $e$ in the polar decomposition of $x$, we can easily check that $x$ is a positive element in the C$^*$-algebra $(A_2(e),\bullet_e,*_e)$, namely, $ee^* (e|x|^{\frac12}) e^*e = e|x|^{\frac12} = (e|x|^{\frac12})^{*_e}$, $(e|x|^{\frac12})\bullet_{e} (e|x|^{\frac12}) = e|x| = x$, and hence $x$ is positive in  $(A_2(e),\bullet_e,*_e)$. Finally, given $y\in J$ the conditions $x\perp y,$ $e e^* = LP(x)$ and $e^* e= RP(x)$ imply that $y\perp e$, and therefore $J\subseteq A_0(e)$. 
\end{proof}

\begin{corollary}\label{c range tripotents in weakly Rickart C* algebras} Let $A$ be a C$^*$-algebra. Then $A$ is a weakly Rickart C$^*$-algebra if, and only if, given $x\in A$ there exists a partial isometry $e$ in $A$ such that $A(x)\subseteq A_2(e),$ $x$ is a positive element in the C$^*$-algebra $(A_2(e), \bullet_e, *_e),$ and $A_0(e) = \{x\}^{\perp}.$
\end{corollary}

\begin{proof} $(\Rightarrow)$ By applying Proposition \ref{p sufficient and necessary conditions in terms of inner ideals without order but local order} to $I= A(x)$ and $J= \{x\}^{\perp}$ we find a partial isometry $e\in A$ satisfying that $I\subseteq A_2(e),$ $e^*e =RP(x)$, $e e^* = LP(x),$ $x$ is a positive element in the C$^*$-algebra $(A_2(e), \bullet_e, *_e),$ $A(x)$ is a C$^*$-subalgebra of $(A_2(e), \bullet_e, *_e),$ and $\{x\}^{\perp}= J\subseteq A_0(e)$.\smallskip

We shall show that $\{x\}^{\perp}= A_0(e)$. To this end take $a\in A_0(e).$ The identities $x a^* = x RP(x) (\mathbf{1}-e^* e) a^* = x (e^*e) (\mathbf{1}-e^* e) a^* =0,$ and $ a^* x= a^* (\mathbf{1}-e e^*) LP(x) x = a^* (\mathbf{1}-e e^*) (e e^*) x =0,$ show that $a\in \{x\}^{\perp}.$ \smallskip

$(\Leftarrow)$ This is a clear consequence of Proposition \ref{p sufficient and necessary conditions in terms of inner ideals without order but local order}, since for each $x\in A$ and each inner ideal $J\subseteq A$ with $I = A(x)\perp J $, by taking the partial isometry $e$ given by the hypothesis we have $J\subset \{x\}^{\perp} = A_0(e)$ and $A(x)\subseteq A_2(e).$ 
\end{proof}

We have seen in the proof of Proposition \ref{p sufficient and necessary conditions in terms of inner ideals without order but local order} that, as a consequence of the result by P. Ara and D. Goldstein \cite{AraGold1993,Gold95}, weakly Rickart C$^*$-algebras satisfy polar decomposition. It is well known that the partial isometry appearing in the polar decomposition of an element $a$ is uniquely determined by $|a|$ (cf. \cite[Propositions 21.1 and 21.3]{BerBaerStarRings}). We shall conclude this section by showing that the properties of the partial isometry $e$ in Corollary \ref{c range tripotents in weakly Rickart C* algebras} provide a characterization of the partial isometry in the polar decomposition. 

\begin{corollary}\label{c characterization of the pi in the polar decomposition} Let $x$ be an element in a weakly Rickart C$^*$-algebra $A$. Suppose $e$ is a partial isometry in $A$. Then the following are equivalent:
\begin{enumerate}[$(a)$]\item $e$ is the partial isometry in the polar decomposition of $x;$
\item $x$ is a positive element in the C$^*$-algebra $(A_2(e), \bullet_e, *_e),$ and $A_0(e) = \{x\}^{\perp}.$ 
\end{enumerate}
\end{corollary}

\begin{proof} The implication $(a)\Rightarrow (b)$ has been proved in the proof of  Corollary \ref{c range tripotents in weakly Rickart C* algebras}.\smallskip

$(b)\Rightarrow (a)$ Since $e$ is a partial isometry, the elements $ee^*$ and $e^* e$ are projections in $A$. It is known that $ee^* A ee^*$ and $e^* e A e^*e$ are Rickart C$^*$-algebras (cf. \cite[Proposition 5.6]{BerBaerStarRings}). Since the mapping $z\mapsto z e^*$ (respectively, $z\mapsto e^* z$) is a C$^*$-isomorphism from $(A_2(e), \bullet_e, *_e)$ onto $ee^* A e e^*$ (respectively, $e^* e A e^*e$), we derive that $(A_2(e), \bullet_e, *_e)$ is a Rickart C$^*$-algebra. \smallskip

We shall next show that the left and right projections of $x$ in $A_2(e)$ both coincide with $e$. Since $x$ is positive in $A_2(e),$ we have $RP_{A_2(e)} (x)= LP_{A_2(e)} (x) =q.$ Clearly $q\leq e$ in $A_2(e).$ If $q<e,$ the partial isometry (projection in $A_2(e)$) $e-q$ is orthogonal to $q$ in $A_2(e)$ and also in $A,$ because orthogonality in $A$ can be given in terms of the triple product $\{a,b,c\} = \frac12 ( ab^* c + c b^* a)$ and $A_2(e)$ is closed for this triple product (see section \ref{sec: JB*-triples} for additional details). When the triple product is computed with respect to the C$^*$-product of $A_2(e)$  and with respect to the one in $A$ we have $$x = \{q, x, q\}_{A_2(e)} = q \bullet_e x^{*_e} \bullet_e q =  q e^* e x^* e e^* q = q x^* q = \{q,x,q\}.$$  It follows that $x$ belongs to $A_2(q)$, which combined with the fact $e-q\perp q,$ implies that $x\perp e-q.$ It follows from the hypotheses that $e-q\in A_0(e).$ Therefore $e-q = e\bullet_e (e-q)= e e^* (e-q)= 0,$ leading to a contradiction.\smallskip

Since $x$ is positive in $A_2(e)$, $RP_{A_2(e)} (x)= LP_{A_2(e)} (x) =e$ in this C$^*$-algebra, and the mapping $z\mapsto e^* z$ is a C$^*$-isomorphism from $(A_2(e), \bullet_e, *_e)$ onto $e^* e A e^* e$, we deduce that $e^* x$ is a positive element in $A$ with  $ e e^* = LP_{A_2(e)} (x) e^* = LP (e^* x)$. Similarly, $e^* e = e^* RP_{A_2(e)} (x) = RP (e^* x )$ (have in mind that the left and right projections of $x e^*$ and $e^* x$ do not change when computed in $A$ or in $ee^* A ee^*$ or $e^* e A e^* e,$ respectively \cite[Proposition 5.6]{BerBaerStarRings}).  Furthermore, since $$\left((e^* x)^* (e^* x) \right)^n = \left(x^* e e^* x \right)^n = (x^* x)^n, \hbox{ for all natural } n,$$ it can be deduced, via functional calculus, that $|x| = e^* x.$\smallskip

It clearly follows from the hypotheses that $x = e e^* x= e |a|$. We have seen above that $RP (e^* x ) = e^* e$ and $ e e^* = LP (e^* x)$. Therefore $e$ is the partial isometry in the polar decomposition of $x$ by uniqueness. \end{proof}

\section{Jordan counterparts of Rickart and Baer *-algebras in terms of projections}\label{subsec: algebraic Jordan Rickart and Baer}

Sh.A. Ayupov and F.N. Arzikulov developed a deep study on the notions of Rickart and Baer $^*$-rings in the setting of real Jordan algebras in the papers \cite{AyuArzi2016Rickart, AyuArzi2019Rickart, Arzi98, Arzi98typeI}. Before entering into details, we introduce the required nomenclature. \smallskip

Let $M$ be a Jordan algebra. According to the standard notation (see \cite{AyuArzi2016Rickart, PeRu2014}), the (\emph{outer}) \emph{quadratic annihilator} of a subset $\mathcal{S}\subset M$ is the set \begin{equation}\label{def quadratic annihilator} \hbox{Ann} (\mathcal{S}) = \mathcal{S}^{\perp_q} :=\{ a\in M : U_a (\mathcal{S}) = \{0\} \}.\end{equation}

The \emph{inner quadratic annihilator} of $\mathcal{S}$ is formed by the elements in the intersection of all kernels of all $U$-maps associated with elements in $S$ defined by \begin{equation}\label{def second quadratic annihilator pre} ^{\perp_q}S :=\{ a\in
	M : U_s (a) = 0 \hbox{ for all } s\in S\}.\end{equation} Let us denote $M^2 := \{ a^2 : a \in M\}$ for the set of all elements in $M$ which are the square of another element (do not confuse with the set of all elements of the form $a\circ b$ with $a, b\in M$). Clearly, each idempotent in $M$ is inside $M^2$.  We consider the following two statements:

\begin{enumerate}[$(R1)$]
	\item For each element $a\in M^2$ there exists an idempotent $e \in M$ (i.e. $e^2 = e$) such
	that $\{a\}^{\perp_q}= U_e(M)$;
	\item For each element $x\in M$ there exists an idempotent $e \in M$ such that $^{\perp_q}\{x\}\cap M^2  =  U_e (M)\cap M^2.$
\end{enumerate} 

In any Jordan algebra $M$, $(R1)$ implies $(R2)$ and both properties are equivalent when $M$ is unital and lacks of nilpotent elements (cf. \cite[Theorems 1.6 and 1.7]{AyuArzi2016Rickart}). According to \cite{AyuArzi2016Rickart, AyuArzi2019Rickart}, a Jordan algebra $M$ satisfying condition $(R1)$ (respectively, $(R2)$) is called a \emph{Rickart Jordan algebra} (respectively, an \emph{inner Rickart Jordan algebra}). That is, each Rickart Jordan algebra is an inner Rickart Jordan algebra. It should be noted here that in \cite{AyuArzi2019Rickart}\label{eq change of notation} inner Rickart Jordan algebras are called \emph{weak Rickart Jordan algebras}, however since the term weak Rickart algebra is employed in the associative setting with another meaing (for example, for an uncountable set $\Gamma$ the commutative C$^*$-algebra $\ell_{\infty, c}(\Gamma)$ of all countably supported elements of the commutative von Neumann algebra $\ell_{\infty, c}(\Gamma)$ is weak Rickart but not an inner Jordan Rickart algebra see, for example, \cite{BerBaerStarRings}), here we shall employ the term mentioned above.\smallskip

The notion of (inner) Rickart is essentially addressed to real JB-algebras. For example, the exceptional JB-algebra $H_3(\mathbb{O})$ is a Rickart Jordan algebra (cf. \cite[Proposition 3.4]{AyuArzi2016Rickart}). Moreover, for each associative Rickart $^*$-algebra $A,$ its self-adjoint part $A_{sa}$ is a Jordan algebra satisfying the $(R1)$ and $(R2)$ (cf. \cite[Proposition 1.1]{AyuArzi2016Rickart}). Reciprocally, if $A$ is an associative $^*$-algebra with proper involution and $A_{sa}$ is Rickart Jordan algebra, then $A$ is a Rickart $^*$-algebra in the usual sense (\cite[Proposition 1.3]{AyuArzi2016Rickart}).\smallskip 

Every Rickart Jordan algebra possesses a unit element and lacks of nilpotent elements, it is further known that the set of idempotents of a Rickart Jordan algebra is 
a lattice, which is not, in general, complete (see \cite[Lemma 1.4, Proposition 1.10]{AyuArzi2016Rickart}). \smallskip

There exist examples of inner Rickart Jordan algebras without unit element (cf. \cite[Remark 1 in page 32]{AyuArzi2019Rickart}). However the properties gathered in the next lemma hold:\smallskip

\begin{lemma}\label{l inner R J unit for squares}\cite[Lemma 2.3]{AyuArzi2019Rickart} Let $M$ be an inner Rickart Jordan algebra. Then the following statements hold:\begin{enumerate}[$(a)$]\item There exists an element $\11_2$ in $M$ satisfying $a\circ \11_2 =  a$ for every $a \in M^2$;
\item $M^2$ contains no non-trivial nilpotent elements.
	\end{enumerate} 
\end{lemma}

The element $\11_2$ given in the above statement $(a)$ is a unit for those elements in $M^2$. If $M$ is generated by square elements (i.e., every element is a finite linear combination of elements in $M^2$), then the element $\11_2$ actually is a unit in $M$.\smallskip

\begin{corollary}\label{c inner Rickart are Rickart for positively generated with no nilpotents}{\rm \cite[Theorems 1.6 and 1.7]{AyuArzi2016Rickart}} Suppose $M$ is a Jordan algebra linearly generated by $M^2$ and containing no non-trivial nilpotent elements. Then $M$ is a Rickart Jordan algebra if and only if it is an inner Rickart Jordan algebra. 
\end{corollary}

The lacking of associativity in Jordan algebras is somehow compensated with the celebrated \emph{Macdonald's theorem} asserting that if $G$ is a multiplication operator in two variables $x$, $y$ with $G(a, b) = 0$ for all $a, b$ in all special Jordan algebras, then  $G = 0$ in all Jordan algebras, equivalently, any polynomial identity in three variables, with degree at most $1$ in the third variable, and which holds in all special Jordan algebras, holds in all Jordan algebras (cf. \cite[Theorem 2.4.13]{HOS}). The following identities, which hold true for any Jordan algebra $M$, can be directly deduced from Macdonald's theorem:
\begin{equation}\label{eq Ta composed with U}
	2 T_{a^l}  U_{a^m, a^n} = 2 U_{a^m, a^n} T_{a^l} = U_{a^{m+l},a^n} + U_{a^m, a^{n+l}}, 
\end{equation}
\begin{equation}\label{eq U_a powers}
	U_{a}^n = U_{a^n}, 
\end{equation} for every natural numbers $l,m,n$ (see \cite[Lemma 2.4.21]{HOS}).\smallskip

In the set of all idempotents in a Jordan algebra $M$ we can consider a partial order defined by $e\leq f$ if $e\circ f = e$. The following equivalences can be easily checked by applying \eqref{eq Ta composed with U} and \eqref{eq U_a powers}:
\begin{equation}\label{eq equivalences of the partial order} e\leq f \Leftrightarrow e \in U_f (M) \Leftrightarrow U_e(M) \subseteq U_f (M). 
\end{equation}

A Jordan algebra $M$ is called a \emph{Baer Jordan algebra} if it satisfies the following property: For each subset $\mathcal{S}\subset M^2$ there exists an idempotent $e \in M$ such that $\mathcal{S}^{\perp_q}= U_e(M)$. We say that $M$ is an \emph{inner Baer Jordan algebra} if for each subset $\mathcal{S}\subset M$ there exists an idempotent $e \in M$ such that $^{\perp_q}\mathcal{S}\cap M^2  =  U_e(M)\cap M^2.$\smallskip

Let us observe that in \cite{AyuArzi2016Rickart, AyuArzi2019Rickart, Arzi98, Arzi98typeI, Arzi99} inner Baer Jordan algebras are called \emph{weak Baer Jordan algebras}, which is a term not completely compatible with the notation in the associative setting.\smallskip

Each Baer Jordan algebra is an inner Baer Jordan algebra \cite[Theorem 2.6]{AyuArzi2016Rickart} or \cite[Proposition 3.1]{AyuArzi2019Rickart}. If $M$ is a Jordan algebra containing no nilpotent elements, then $M$ is an inner Baer Jordan algebra if and only if it is a Baer Jordan algebra \cite[Theorem 2.6]{AyuArzi2016Rickart}. As we have seen in the comments after Lemma \ref{l inner R J unit for squares}, if a Jordan algebra $M$ is linearly generated by elements in $M^2$ and $M$ is an inner Baer Jordan algebra, then $M$ is unital. A C$^*$-algebra is a Baer C$^*$-algebra if and only if $A_{sa}$ is a Baer Jordan algebra (cf. \cite[Propositions 2.1 and 2.3]{AyuArzi2016Rickart} or \cite{AyuArzi2019Rickart}).\smallskip

To conclude our tour through the algebraic Jordan alter-egos of Rickart and Baer algebras, we appeal to a couple of results also proved by  Sh.A. Ayupov and F.N. Arzikulov where they establish that a Jordan algebra $M$ is a Baer Jordan algebra if, and only if, it is a Rickart Jordan algebra and the set of all idempotents in $M$ is a complete lattice (see \cite[Theorem 2.7]{AyuArzi2016Rickart}); moreover, $M$ is an inner Baer Jordan algebra if, and only if, it is an inner Rickart Jordan algebra and the set of all idempotents of $M$ is a complete lattice (cf. \cite[Theorem 3.5]{AyuArzi2019Rickart}).\smallskip


Following \cite{AyuArzi2016Rickart, AyuArzi2019Rickart, Arzi98, Arzi98typeI}, and in coherence with the terminology of C$^*$-algebras, (\emph{inner}) \emph{Rickart JB-algebras} and (\emph{inner}) \emph{Baer JB-algebras} or \emph{AJBW-algebras} are defined as those JB-algebras which are (\emph{inner}) Rickart and (\emph{inner}) Baer Jordan algebras, respectively. We shall also deal with the complex structures. A JB$^*$-algebra $M$ will be called a \emph{Rickart JB$^*$-algebra} (respectively, a \emph{Baer JB$^*$-algebra} or an {AJBW$^*$-algebra}) if it self-adjoint part, $M_{sa},$ is a Rickart JB-algebra (respectively, a {Baer JB-algebra} or an {AJBW-algebra}). That is, $M$ is a Rickart JB$^*$-algebra if and only if for each $a\in M^+$ there exists a projection $p \in M$ such that $$\{a\}^{\perp_q} \cap M_{sa} = U_p(M)\cap M_{sa} = Q_p(M)\cap M_{sa};$$ which by Corollary \ref{c inner Rickart are Rickart for positively generated with no nilpotents} is equivalent to prove that for each $x\in M_{sa}$ there exists a projection $p \in M$ such that $$^{\perp_q}\{x\}\cap M^+  =  U_p (M)\cap M^+ = Q(p) (M)\cap M^+ .$$ A similar restatement can be applied to the definition of Baer JB$^*$-algebras. A JBW$^*$-algebra (respectively, a JBW-algebra) is a JB$^*$-algebra (respectively, a JB-algebra) which is a dual Banach space. It is known that a JB$^*$-algebra $M$ is a JBW$^*$-algebra if, and only if, $M_{sa}$ is a JBW-algebra (cf., for example, \cite[Corollary 2.12]{MarPe2000}). \smallskip

Two elements $a,b$ in a Jordan algebra $A$ are said to operator commute if $$ a\circ (b\circ x)=(a\circ x)\circ b$$ for every $x\in A.$ By the mentioned \emph{Macdonald's theorem} or by the \emph{Shirshov-Cohn theorem} \cite[Theorem 2.4.14]{HOS}, it can be easily checked that operator commutativity of a couple of elements in a Jordan algebra can be equivalently verified in any Jordan subalgebra containing these elements (cf. \cite[Proposition 1]{Topping}).\smallskip

A real Jordan algebra $N$ is called \emph{formally real} if for every $a_1,\ldots, a_n\in N$ the condition $\sum_{i=1}^n a_i ^2 =0$ implies $a_1=\ldots= a_n=0 $ (see \cite[\S 2.9]{HOS}). Every JB-algebra is a formally real Jordan algebra. A Jordan subalgebra $B$ of $N$ is called \emph{strongly associative} if the identity $(x\circ y)\circ a = x\circ  (y\circ a)$ holds for all $x, a \in  B$ and $y \in N$, equivalently, any pair of elements in $B$ operator commute as elements in $N$. A family $\mathcal{F}$ of elements of $N$ is called \emph{compatible} if the Jordan subalgebra $J(\mathcal{F})$ generated by $\mathcal{F}$ is strongly associative.\smallskip

The idea behind (weakly) Rickart and Baer C$^*$-algebra is is to find a subclass of C$^*$-algebras, between general C$^*$-algebras and von Neumann algebras, in which every element can be approximated in norm by finite linear combinations of projections. In the setting of AJBW$^*$-algebras (i.e. Baer JB$^*$-algebras) this goal is achieved by the following theorem, in which Arzikulov established a Jordan version of the original result proved by Kaplansky for AW$^*$-algebras. 

\begin{theorem}\label{t AJBW-}\cite[Theorem 2.1]{Arzi98}
The following statements are equivalent for each JB-algebra $N$:
\begin{enumerate}[$(a)$]\item $N$ satisfies the following properties:
\begin{enumerate}[$(1)$]\item Every subset of pairwise orthogonal projections in a partially ordered set of projections has a least upper bound in this set;
\item  Every maximal strongly associative subalgebra of $N$ is generated by its projections (i.e., it coincides with the least closed subalgebra containing its projections);
\end{enumerate}
\item $N$ is an AJBW-algebra;
\item $N$ is an inner AJBW-algebra.
\end{enumerate}
\end{theorem}

Let $M$ be a JB$^*$-algebra. It is worth to notice that the JB$^*$-subalgebra generated by a single self-adjoint element in $M$ is strongly associative (cf. \cite[Proposition 2.4.13 and Fact 3.3.34]{Cabrera-Rodriguez-vol1}). The set of all strongly associative subalgebras of $M$ can be regarded as an inductive set when equipped with the order defined by inclusion. Therefore each strongly associative JB$^*$-subalgebra of $M$ is contained in a maximal strongly associative JB$^*$-subalgebra. It follows from Theorem \ref{t AJBW-} that every self-adjoint element in a AJBW$^*$-algebra $M$ can be approximated by finite linear combinations of projections in $M$ --actually the same conclusion holds for any element in $M$.  We shall see later that our notion of weakly Rickart JB$^*$-algebra also enjoys this property. 
\smallskip

As in the case of C$^*$-algebras, a couple of projections $p,q$ in a JB$^*$-algebra are called \emph{orthogonal} if $p\circ q =0$. Both notions are perfectly compatible in the case of a C$^*$-algebra regarded with its associative structure or as a JB$^*$-algebra.\smallskip

One of the new contributions in this note is to explore the notions of weakly Rickart and SAW$^*$-algebras in the setting of JB$^*$-algebras. In order to develop our study, we shall follow a similar method to that introduced by Ayupov and Arzikulov focused on the self-adjoint part and the lattice of projections.  In the setting of JB$^*$-algebras we cannot define properties in terms of the left or right multiplication by an element. We gather next some reinterpretations for latter purposes.

\begin{lemma}\label{l annihilator for Jordan products and for the associative product with positive elements} Let $a$ and $x$ be non-zero positive elements in a C$^*$-algebra. Then the following statements are equivalent:\begin{enumerate}[$(a)$]
		\item $a x = x;$
		\item $a\circ x = x;$
		\item $U_a(x) = x$.
\end{enumerate} Clearly, the elements $a$ and $x$ commute in case that any of the previous statements holds. 
\end{lemma} 

\begin{proof} $(a)\Rightarrow (b)$ and $(c)$. This is clear because $x a = (a x)^* = x^* = x$, and thus $a\circ x = \frac12 ( a x + x a) = x$. \smallskip
	
Similarly, $U_a (x) = a x a = x a = x.$\smallskip

$(b) \Rightarrow (a)$ We can clearly embed $A$ inside its unitization, and thus assume that $A$ is unital. Since $(1-a) \circ x = 0$ with $1-a\in A_{sa}$ and $x\geq 0$, \cite[Lemma 4.1]{BurFerGarPe09} implies that $x\perp (1-a)$ in $A$ (as JB$^*$- and as C$^*$-algebra), then $(1-a) x =0 = x (1-a)$, which proves $(a)$.\smallskip

$(c) \Rightarrow (a)$ If $\|a\|\leq 1$ the proof is much easier. First, the inequality $\|x\| =\|U_{a} (x)\| \leq \|a\|^{2} \ \|x\|$ assures that $\|a\| =1$. We can deduce from a simple induction argument that $U_{a^n} (x) = a^n x a^n = x$ for all natural $n$. Now, by applying that the sequence $(a^n)_n$ converges in the strong$^*$ topology of $A^{**}$ to the support projection, $s(a)$, of $a$, together with the join strong$^*$ continuity of the product of $A^{**}$ \cite[Proposition 1.8.12]{Sa}, we obtain $s(a) x s(a) =x$. Finally, since $a = s(a) + (1-s(a)) a = s(a) + a (1-s(a))$ in $A^{**}$, it follows that $a x = s(a) x + a (\11-s(a)) x = x$.\smallskip

For the general case we assume that $a x a =x.$ Since the same identity holds in $A^{**}$, it is easy to check that $a z a = z$ for every $z$ in the C$^*$-subalgebra of $A$ generated by $x$ (and also in the von Neumann subalgebra of $A^{**}$ generated by $x$). Therefore, the identity $a\ r(x)\ a = r(x)$ holds in $A^{**}.$ It is easy to deduce from the above that $$ (r(x)\ a\ r(x)) \ ( r(x)\ a\ r(x)) = r(x).$$ Having in mind that $r(x)\ a\ r(x)$ is a positive element with $r(x)\ a\ r(x) \leq \|a\| \ r(x)$ whose square is $r(x)$, a simple application of the local Gelfand theory proves that $ r(x)\ a\ r(x) =r(x).$ \smallskip

Now by mixing the identities $a\ r(x)\ a = r(x)$ and $ r(x)\ a\ r(x) =r(x)$ we get $$a\ r(x) =(a \  r(x)\ a)\ r(x) = r(x),\hbox{ and } r(x) \ a = r(x)\ (a\ r(x)\ a)= r(x).$$  Finally, it is easy to see that $ a x = a r(x) x = r(x) x = x = x r(x) = x r(x) a=xa.$ 
\end{proof}

As in the associative setting of C$^*$-algebras, a JB$^*$-subalgebra $B$ of a JB$^*$-algebra $M$ is said to be an hereditary  JB$^*$-subalgebra of $M$ if whenever $0 \leq a\leq  b$ with $a \in  M$ and $b \in B$, then $a\in B$, equivalently, $B^+$ is a face of $M^+$  (cf.  \cite{Edw77, Bu1983, AlfsenShultz2003}).\smallskip

It is known that an hereditary C$^*$-subalgebra $B$ of C$^*$-algebra $A$ is $\sigma$-unital if and only if it has the form $B= \overline{xAx}$ for some positive $x\in A$. The same statement remains valid in tha case of a JB$^*$-algebra $M$, where each $\sigma$-unital, hereditary JB$^*$-subalgebra is of the form $\overline{U_x(M)}$, for some positive $x\in M$. \smallskip

\begin{corollary}\label{c unit for the Jordan product} Let $a$ and $x$ be positive elements in a JB$^*$-algebra $M$. Then the following statements are equivalent:\begin{enumerate}[$(a)$]
		\item $a\circ x = x;$
		\item $U_a(x) = x$;
		\item $a \circ z = z$ for all $z$ in the inner ideal of $M$ generated by $x$.
	\end{enumerate} Furthermore, if any of the previous statements holds the elements $a$ and $x$ operator commute as elements of $M,$ and $a\circ r(x) = r(x),$ where $r(x)$ denotes the range projection of $x$ in $M^{**}$. 
\end{corollary}

\begin{proof} By Macdonald's theorem (see also the Shirshov-Cohn theorem in \cite{HOS} or \cite[Corollary 2.2]{Wright77}), there exists a C$^*$-algebra $A$ containing the JB$^*$-subalgebra of $M$ generated by $a$ and $x$ as JB$^*$-subalgebra. Lemma \ref{l annihilator for Jordan products and for the associative product with positive elements} proves that $(a)$ is equivalent to $(b)$ in $A,$ and hence in $M$. Since $a x = x a = x$ in $A$, \cite[Proposition 1]{Topping} assures that $a$ and $x$ operator commute in $M$.\smallskip
	
The implication $(c)\Rightarrow (a)$ is clear because $x\in M(x)$. 	To see the implication $(a)\Rightarrow (c)$, we recall that $(a)$ implies that $a$ and $x$ operator commute in $M$ and $a x = x a = x$ in $A$ (cf. Lemma \ref{l annihilator for Jordan products and for the associative product with positive elements}). Then  $$\{a,x,z\} = (a\circ x) \circ z + (z\circ x) \circ a - (a\circ z) \circ x = x\circ z \ \ (x\in M),$$ and thus, by the Jordan identity, we get
$$\begin{aligned}
U_a U_x (y) &= \{a,\{x,y^*,x\},a\} = -\{y^*,x,\{a,x,a\}\} + 2 \{\{y^*,x,a\},x,a\} \\
&= -\{y^*,x,x\} + 2 (x\circ y^*) \circ x\\
&=-  (x\circ y^*) \circ x -x^2\circ y^* +(x\circ y^*) \circ x + 2 (x\circ y^*) \circ x\\
&= \{x,y^*,x\}=U_x (y),
\end{aligned}
$$ for all $y\in M$.  This shows that $U_a(z)=z$ for every $z\in M(x)$. Now take $z\in M(x)$ positive, then, by the equivalence $(a)\Leftrightarrow (b)$, $U_a(z)=z$ gives $a\circ z=z$. Since, each $z\in M(x)$ writes as a linear combination of four positive elements in $M(x)$, we have $a\circ z=z$.
\end{proof}

The weak versions of Rickart and Baer Jordan algebras in the classical sense considered in Berberian's book \cite{BerBaerStarRings} have not been considered yet.  The reader should be warned that, in order to work in the Jordan setting, the left and right multiplication operations do not make too much sense in a Jordan algebra. 

\begin{definition}\label{def weak Rickart} Let $N$ be a JB-algebra.
\begin{enumerate}[$\checkmark$]\item We shall say that $N$ is a {weakly Rickart JB-algebra} if for each element $a\in N^+$ there exists a projection $p \in N$ such that $p\circ a = a,$ and for each $z\in N$ with  $U_z(a) =0$ we have $p\circ z =0$. 
		
\item $N$ is called a \emph{weakly inner Rickart JB-algebra} if for each element $x\in N$ there exists a projection $p \in N$ such that $p\circ x = x,$ and for each $z\in N^+$ with $U_x(z )= 0$ we have $p\circ z =0$.
\end{enumerate}
\end{definition}

A JB$^*$-algebra $M$ will be called weakly Rickart or weakly inner Rickart if its self-adjoint part satisfies the same property.\smallskip

\begin{remark}\label{r existence and uniqueness of the range projection in wRickart JBalgebras}{\rm Let $N$ be a weakly Rickart JB-algebra. Then, for each $a\in N^+$, the projection $p$ in Definition \ref{def weak Rickart} is unique. This projection will be called the \emph{range projection} of $a$ in $N$ (RP$_{N} (a)=$RP$(a)$ in short).  Indeed, suppose that there exist projections $p, p'$ in $N$ such that 
$$p\circ a=p'\circ a=a,$$
and for any $z\in N$ with $U_z(a)=0$ we have $p\circ z=p'\circ z=0$. It follows from the original assumptions that $(p-p')\circ a=0$, and since $a\geq 0$, we deduce from \eqref{eq equivalence of orthogonality and zero product with a positive factor} that $p-p' \perp a$. It then follows that $U_{(p-p')} (a) = \{p-p', a, p-p'\}=0.$ By applying the assumptions we get $p\circ (p-p') = 0 = p'\circ (p'-p),$ which implies that $p=p\circ p' = p'$.\smallskip

It can be seen that $RP_N(a)$ is the smallest projection in $N$ such that $a = p\circ a (= U_p (a)).$ Namely, if $q$ is any projection in $N$ such that $q \circ a = a,$ then $(RP_N(a)-q)\circ a = 0$, and thus $RP_N(a) \circ (RP_N(a)-q) = 0$, therefore $RP_N(a) \circ q = RP_N(a)$, which is equivalent to say that $RP_N(a)\leq q.$}
\end{remark}

\begin{lemma}\label{l weak Rickart and weak Baer} Let $N$ be a JB-algebra. Then $N$ is weakly Rickart and unital if, and only if, it is a Rickart JB-algebra if, and only if, it is weakly inner Rickart and unital.
\end{lemma}

\begin{proof} Suppose $N$ is a unital weakly Rickart JB-algebra with unit $\textbf{1}$. Let us fix $a\in N^+$. By assumptions there exists a projection $p$ in $N$ such that $a\circ p = a$ and for each $z\in N$ with $U_z(a) =0$ we have $p\circ z= 0$. Given $x\in \{a\}^{\perp_q}$ we have $U_x (a) =0$, and thus $p\circ x = 0$, in particular $(\textbf{1}-p)\circ x = x$. We have shown that $\{a\}^{\perp_q}\subseteq U_{\textbf{1}-p}(N) = N_2(\textbf{1}-p)= N_0(p)$. Conversely, if $x\in U_{\textbf{1}-p}(N)= N_2(\textbf{1}-p)= N_0(p)$, since $p\circ a = a$, we deduce that $a\in N_2(p),$ and consequently, $U_x (a) = 0,$ by Peirce arithmetic.  \smallskip
	
Suppose now that $N$ is a unital weakly inner Rickart JB-algebra with unit $\textbf{1}$.	So, given $x\in M$ there exists a projection $p \in N$ such that $p\circ x = x$ and for each $z\in N^+$ with $U_x(z)=0$ we have $p\circ z= 0$. For each $z\in ^{\perp_q}\{x\}\cap N^2$ we have $U_x (z)=0$, and hence $p\circ z=0$. It follows that $^{\perp_q}\{x\}\cap N^2 \subseteq  U_{\textbf{1}-p} (N)\cap N^2$. Reciprocally, each $z\in U_{\textbf{1}-p} (N)\cap N^2$ is positive and must be orthogonal to $N_2(p)$ by Peirce arithmetic, then $z\in\  ^{\perp_q}\{x\}\cap N^2,$ because $x\in N_2(p)$. \smallskip

To conclude the proof we observe that every Rickart JB-algebra is unital and weakly (inner) Rickart. 
\end{proof}

\begin{proposition}\label{p weak Rickart condition is inherited by Peirce2 projections} Let $p$ be a projection in a weakly Rickart JB$^*$-algebra $M$. Then the Peirce-2 subspace $M_2(p)$ is a Rickart JB$^*$-algebra with unambiguous range projections of positive elements in $M_2(p)$.
\end{proposition}

\begin{proof} Let us fix a positive element $a\in M_2(p).$ Clearly, $a$ is positive in $M$. Let $q= \hbox{RP}(a)$ denote the range projection of $a$ in $M$. Since $(p-q)\circ a = 0,$ it follows from \eqref{eq equivalence of orthogonality and zero product with a positive factor} that $(p-q)\perp  a$, and thus $U_{(p-q)} (a) =0.$ Applying now that $q=\hbox{RP} (a)$ we get $q\circ (p-q) =0.$ Therefore, $p\circ q =q$ and thus $U_{p} (q) = q,$ witnessing that $q\in M_2(p)$ and satisfies the properties of a range projection for $a$ in $M_2(p).$  We have proved that $M_2(p)$ is a unital weakly Rickart JB$^*$-algebra, Lemma \ref{l weak Rickart and weak Baer} gives the rest.  
\end{proof}

Let $h$ and $x$ be two elements in a JB$^*$-algebra $M$ with $h$ positive. We know from \cite[Lemma 4.1]{BurFerGarPe09} that  \begin{equation}\label{eq equivalence of orthogonality and zero product with a positive factor} \hbox{$x \perp  h$ if, and only if, $h \circ x = 0$.}    
\end{equation}

The \emph{orthogonal annihilator} of a subset $\mathcal{S}$ in a JB$^*$-triple $E$ is defined as $$
\mathcal{S}_{_E}^\perp = \mathcal{S}^{\perp}:=\{ y \in E : y \perp x , \forall x \in \mathcal{S} \}.$$ 

The next result with the basic properties of the orthogonal annihilator has been borrowed from \cite[Lemma 3.1]{BurGarPe11StudiaTriples} and \cite[Lemma 3.2]{EdRu96}.

\begin{lemma}\label{l basic prop annihilator}{\rm(\cite[Lemma 3.2]{EdRu96}, \cite[Lemma 3.1]{BurGarPe11StudiaTriples})} Let $\mathcal{S}$ be a nonempty subset of a JB$^*$-triple $E$. Then the following statements hold: \begin{enumerate}[$(a)$]
\item $\mathcal{S}^{\perp}$ is a norm closed inner ideal of $E$;
\item $\mathcal{S}\cap \mathcal{S}^{\perp}= \{0\};$
\item $\mathcal{S}\subseteq \mathcal{S}^{\perp \perp};$
\item If $\mathcal{S}_1\subseteq \mathcal{S}_2$ then $\mathcal{S}_2^{\perp}\subseteq \mathcal{S}_1^{\perp};$
\item $\mathcal{S}^{\perp}$ is weak$^*$ closed whenever $E$ is a JBW$^*$-triple.
\end{enumerate}
\end{lemma}

We should note that the orthogonal annihilator of a subset $\mathcal{S}$ in a JB$^*$-algebra $M$ need not coincide with the quadratic annihilators defined in \eqref{def quadratic annihilator} and \eqref{def second quadratic annihilator pre}. In general we have \begin{equation}\label{eq the orthogonal annihilator is smaller} (\mathcal{S}^{\perp})^* \subseteq \  ^{\perp_q}\mathcal{S}, \ \hbox{ and }  (\mathcal{S}^*)^{\perp} \subseteq \mathcal{S}^{\perp_q}, \hbox{ for all } \mathcal{S}\subset M,  
\end{equation} where $\mathcal{S}^* =\{ x^* : x\in \mathcal{S} \}$. The equalities do not necessarily hold. For example, let $e$ be a complete tripotent in $M= B(H)$ which is not unitary (for example a partial isometry satisfying $e e^* = \11$ and $p= e^* e \neq \11$). Clearly, $\{e\}^{\perp} = M_0(e) = \{0\}$ and $\{e^*\}^{\perp} = M_0(e^*) = \{0\}.$ It is easy to check that $^{\perp_q}\{e\} = M_1(e) = M (\11-p)$ and $(\11-p) M \subseteq  \{e\}^{\perp_q}$.

\begin{lemma}\label{l positive sets and annihilators} Let $\mathcal{S}$ be a set of positive elements in a JB$^*$-algebra $M$. Then $$S^{\perp_q}\cap M_{sa} = S^{\perp} \cap M_{sa}, \hbox{ and  } ^{\perp_q}S\cap M^+ = S^{\perp} \cap M^{+}.$$  
\end{lemma}

\begin{proof} The inclusion $\supseteq$ is clear from \eqref{eq the orthogonal annihilator is smaller}. Fix $s\in \mathcal{S}$ and $h\in S^{\perp_q}\cap M_{sa}$. We can find, via Macdonald's or Shirshov-Cohn theorem, a C$^*$-algebra $B$ containing $s$ and $h$ as positive and hermitian elements, respectively. Since $s = b^2$ for some $b\in M$ and also in $B$, and $0 = U_h (s) = U_{h} (b^2) = (h b) (b h)^*$, we deduce that $h b = b h =0,$ and hence $h s = h b^2 =0$ and $h\circ s =0$ in $B$ and in $M$. This is enough to guarantee that $h\perp s$ (cf. \cite[Lemma 4.1]{BurFerGarPe09}).  The other equality can be proved similarly. 
\end{proof}

The following lemma is probably known, but it is included here for the lacking of an explicit source. 

\begin{lemma}\label{l inner ideals as orthogonal annihilators of sets of positive} Let $\mathcal{S}$ be a subset of positive elements in a JB$^*$-algebra $M$. Then the orthogonal annihilator of $\mathcal{S}$,  $\mathcal{S}^{\perp},$ is a triple inner ideal and a hereditary JB$^*$-subalgebra of $M$. 
\end{lemma}

\begin{proof} Clearly $\mathcal{S}^{\perp}$ is a closed subspace and an inner ideal (see Lemma \ref{l basic prop annihilator}). Let us take $x\in \mathcal{S}^{\perp}$ and $s\in \mathcal{S}$. Since $s\perp x$ with $s\geq 0$, we deduce from \cite[Lemma 4.1]{BurFerGarPe09} that $x\circ s =0,$ and hence $x^* \circ s =0$, which is equivalent to $s\perp x^*,$ and consequently $x^*\in  \mathcal{S}^{\perp}$.\smallskip

The elements $h = \frac{x+x^*}{2}$ and $k = \frac{x-x^*}{2i}$ lie in $M_{sa}\cap \mathcal{S}^{\perp}.$  Since $h^2\circ s = \{h,h,s\} =0$, a new application of \cite[Lemma 4.1]{BurFerGarPe09} proves that $h^2\in \mathcal{S}^{\perp}.$ Similarly, $k^2\in \mathcal{S}^{\perp}.$ Actually, since $h+k\in \mathcal{S}^{\perp},$ we can similarly deduce that $(h+k)^2\in \mathcal{S}^{\perp}.$ It follows from this that $h\circ k \in \mathcal{S}^{\perp},$ and $x^2 = h^2 - k^2 + 2 i h\circ k \in\mathcal{S}^{\perp}$ too.\smallskip

Finally, let us take $0\leq a \leq b$ with $b\in S^{\perp}$ and any $s\in S^{\perp}$ then $0\leq U_s(a)\leq U_s(b)=0$. It follows from Lemma \ref{l positive sets and annihilators} that $a\in \mathcal{S}^{\perp}.$
\end{proof}

Our next definition is now fully justified by the previous results. 

\begin{definition}\label{def SAJBW-algebra}
Let $N$ be a JB-algebra. We shall say that $N$ is a SAJBW-algebra if for any $x,y\in N^+$ with $ x\circ y =0$ there exists $e\in N^+$ (not necessarily a projection) such that $e\circ x = x$ and $e\circ y =0$. A JB$^*$-algebra $M$ will be called a SAJBW$^*$-algebra if its self-adjoint part is a SAJBW-algebra.
\end{definition}
	
The next proposition is a generalization of Pedersen's result in Proposition \ref{Prop Pedersen} to the setting of JB$^*$-algebras. Our new notions of weakly Rickart and SAJBW$^*$-algebras are the missing ingredients to complete the whole picture. 

\begin{proposition}\label{p sufficient and necessary conditions in terms of inner ideals} Let $M$ be a JB$^*$-algebra. Consider the following property: Given two orthogonal, hereditary JB$^*$-subalgebras $B$ and $C$ of $M$, there is a positive $e$ in $M$ which is a unit for $B$ and annihilates $C$. \begin{enumerate}[$(a)$]\item The previous property holds for all pairs of hereditary JB$^*$-subalgebras $B$, $C$ if, and only if, $M$ is an AJBW$^*$-algebra; 
\item The property holds when $B$ is the inner ideal generated by a positive element and $C$ is arbitrary if, and only if, $M$ is a weakly Rickart JB$^*$-algebra;
		\item The property holds when $C$ is the inner ideal generated by a positive element and $B$ is arbitrary if, and only if, $M$ is a Rickart JB$^*$-algebra;  
		\item The property holds when both $B$ and $C$ are inner ideals generated, each one of them by a single positive element if, and only if, $M$ is a SAJBW$^*$-algebra.
	\end{enumerate} 
\end{proposition}

\begin{proof}	$(d)$ $(\Rightarrow)$ By considering two positive elements $x,y$ in $M$ with $x\circ y=0$, the inner ideals $M(x)$ and $M(y)$ are orthogonal, and hence by hypothesis, there exists a positive $e\in M$ which is a unit for $M(x)$ and annihilates $M(y)$. Clearly, $ e\circ x = x$ and $e\circ y =0$.\smallskip
	
$(\Leftarrow)$ If $M$ is a SAJBW$^*$-algebra, given positive elements $x,y$ in $M$ with $x\circ y = 0$, there exists a positive $e$ in $M$ such that $e\circ x = x$ and $e\circ y =0$ --the latter being equivalent to $y\perp e$ by \eqref{eq equivalence of orthogonality and zero product with a positive factor}. Corollary \ref{c unit for the Jordan product} implies that $e$ is a unit for $B=\overline{ U_x(M)}$. Furthermore, for each $a$ in $M$ the elements $e$ and $U_y (a)$ are orthogonal since $y\perp e$ and $\{e\}^{\perp}$ is an inner ideal of $M$ and hence contains all elements in $U_y(M) = Q(y)(M)$. It follows that $e$ annihilates $C=\overline{ U_y(M)}$. \smallskip
	
$(b)$  $(\Rightarrow)$ Fix a positive $a\in M$, by applying the hypothesis to $B= M(a)$ and $C= \{a\}^{\perp}$ we find a positive $e\in M$ which is a unit for $M(a)$ and annihilates $\{a\}^{\perp}$. We know from Corollary \ref{c unit for the Jordan product} that $e$ and $a$ operator commute, and thus $(e^n -e) \circ a = 0$ for all natural $n$. Having in mind \eqref{eq equivalence of orthogonality and zero product with a positive factor}, the properties of $e$ assure that $0 = e\circ (e^n -e) = e^{n+1} -e^2$ for all natural $n$. A simple application of the local Gelfand theory on the commutative and associative JB$^*$-algebra generated by $e$ proves that $e$ is a projection. 
\smallskip

Now we take any $z\in \{a\}^{\perp_q}\cap M_{sa}$. By Lemma \ref{l positive sets and annihilators}, $\{a\}^{\perp_q}\cap M_{sa} = \{a\}^{\perp}\cap M_{sa},$ and thus the properties of $e$ imply that  $e\circ z =0.$ Therefore, $M$ is a weak Rickart JB$^*$-algebra.\smallskip 
	
$(\Leftarrow)$ Suppose now that $M$ is a weak Rickart JB$^*$-algebra. Take $B= \overline{ U_x(M)}$ and $C$ as in the statement, with $x$ positive in $M$. It follows from the hypothesis that there exists a projection $p\in M$ satisfying $p \circ x = x$ and $p\circ z = 0$ for all $z\in M_{sa}$ with $U_z (x)=0.$ Clearly, each $c\in C_{sa}$ satisfies $U_c (x) =0$, and thus $p\circ c =0$ for all $c\in C$.\smallskip

$(c)$  $(\Rightarrow)$ For $C = M(0)=\{0\}$ and $B= M$, the hypothesis implies the existence of a unit element $\11\in M.$ Pick $a\in M^+$. Since $B=\{a\}^{\perp}$ and $C= M(a) $ are two orthogonal hereditary JB$^*$-subalgebras, by hypothesis, there exists a positive $e\in M$ which is a unit for $B$ and annihilates $C$. That is $e\in \{a\}^{\perp},$ and hence $e \circ e =e$, witnessing that $e$ is a projection in $M$.\smallskip 
	
As before, Lemma \ref{l positive sets and annihilators} proves that $\{a\}^{\perp_q}\cap M_{sa} = \{a\}^{\perp}\cap M_{sa}.$ It follows from the properties of $e$ that $\{a\}^{\perp_q}\cap M_{sa}\subseteq \{\11-e\}^{\perp} = U_{e} (M)$. Reciprocally, if $z\in  U_{e} (M)\cap M_{sa},$ since $e\circ a =0,$ and hence $a\in U_{\11-e} (M)$, it follows that $z\in  \{a\}^{\perp}\cap M_{sa} = \{a\}^{\perp_q}\cap M_{sa}.$\smallskip 
	
$(\Leftarrow)$ We assume now that $M$ is a Rickart JB$^*$-algebra. Take $C= \overline{ U_x(M)}$ and $B$ as in the statement, with $x$ positive in $M$. Under these circumstances there exists a projection $p\in M$ satisfying $\{x\}^{\perp_q} \cap M_{sa} = U_p (M)\cap M_{sa}$. Clearly, each $b\in B$ lies in $\{x\}^{\perp_q}$, and thus $p\circ b = b$ for all $b\in B$. Since $p\in U_p(M)$, we have $U_p (x)  =0$. Having in mind that $x$ is positive, we deduce, via Shirshov-Cohn theorem, that $p$ and $x$ are orthogonal. Consequently, by Peirce arithmetic, $p$ annihilates $C= \overline{ U_x(M)}$.\smallskip

	$(a)$  $(\Rightarrow)$ Taking $B=\{0\}$ and $C= A$, we find a unit $\11\in M$.  Fix a subset $\mathcal{S}\subseteq M^+$. The inner ideal $C=\mathcal{S}^{\perp}$ is a hereditary JB$^*$-subalgebra of $M$, and the same happens to $B= \left(C\cap M^+\right)^{\perp}$. Clearly, $B\perp C$. By assumptions, there exists a positive $e$ in $M$ which is a unit for $B$ and annihilates $C$. In particular $\11-e$ lies in $C$ and hence $e\circ (\11-e) = 0$. Thus, $e$ is a projection.\smallskip
	
	Lemma \ref{l positive sets and annihilators} implies that $S^{\perp_q}\cap M_{sa} = S^{\perp} \cap M_{sa}  = U_{\11-e} (M) \cap M_{sa},$ where the last equality follows from the same arguments given in the proof of $(c)$.\smallskip
	
	$(\Leftarrow)$ We assume finally that  $M$ is an AJBW$^*$-algebra (it is, in particular, unital). Taking $B$ and $C$ as in the statement, for $C^+$, there exists a projection $p$ in $M$ such that $(C^+)^{\perp_q} = U_p (M_{sa})$. Since $B^+\subseteq (C^+)^{\perp_q},$ $p$ is the unit element in $ U_p (M_{sa})$, and every element in $B$ is a linear combination of four positive elements in $B$, $p$ must be a unit for $B$. On the other hand, each positive $c\in C$ satisfies that $U_p (c)=0$, and thus $p$ is orthogonal to each positive element in $C$. Therefore, $p$ is orthogonal to $C$, as desired. 
\end{proof}	

Let $A$ be a C$^*$-algebra. It follows from the previous proposition and from Proposition \ref{Prop Pedersen} that $A$ is an AJBW$^*$-algebra (respectively, a Rickart, a weakly Rickart or a SAJBW$^*$-algebra) if and only if it is an AW$^*$-algebra (respectively, a Rickart, a weakly Rickart or a SAW$^*$-algebra). The statement concerning AJBW$^*$-algebras and AW$^*$-algebra (respectively, Rickart JB$^*$-algebras and Rickart C$^*$-algebras) can be derived from the results by Ayupov and Arzikulov in \cite[Propositions 1.1, 1.3, 2.1 and 2.3] {AyuArzi2016Rickart}.\smallskip

The next technical lemma will be required in the main result of this section. Before presenting the result, we recall some facts on operator commutativity. By the Shirshov--Cohn theorem \cite[Theorem 2.4.14]{HOS} any two self-adjoint elements $a$ and $b$ in a JB$^*$-algebra $M$ generate a JB$^*$-subalgebra that can be realised as a JC$^*$-subalgebra of some $B(H)$ (see also \cite[Corollary 2.2]{Wright77}). Furthermore, under this identification, $a$ and $b$ commute in the usual sense whenever they operator commute in $M$ (compare Proposition 1 in \cite{Topping}). By the same arguments, for any pair of self-adjoint elements $a$ and $b$ in $M$ we have
\begin{equation}\label{eq operator comm of self-adjoint} a \hbox{ and } b \hbox{ operator commute if and only if $a^2 \circ b  = 2 (a\circ b)\circ a - a^2 \circ b$} 
\end{equation}

\begin{lemma}\label{l strongly associativity is inherited by the range projection} Let $M$ be a weakly Rickart JB$^*$-algebra. Let $a,b$ be two elements in $M$ with $a$ positive. Suppose that $a$ and $b$ operator commute. Then $RP(a)$ and $b$ operator commute. 
\end{lemma}

\begin{proof} Let $p = RP(a)\in M.$ Let us write, $b = b_1 + i b_2 $, where each $b_j$ is self-adjoint for every $j=1,2$ and $a$ operator commutes with $b_1$ and $b_2$. Let us consider the element $c_j = p\circ b_j - b_j$. Having in mind that $a$ and $b_j$ operator commute and $p = RP(a)$ we obtain $$ c_j \circ a = (p\circ b_j - b_j) \circ a = (p\circ a)\circ b_j - b_j \circ a = a\circ b_j - b_j \circ a =0.$$ Since $a$ is positive, the above identity proves that $a\perp c_j$ (cf. \eqref{eq equivalence of orthogonality and zero product with a positive factor}). It follows from the properties of the range projection that $p\circ c_j =0,$ that is, $p\circ (p\circ b_j - b_j)=0$, or equivalently, $p\circ (p\circ b_j) =  p \circ b_j =  p^2 \circ b_j,$ which is equivalent to say that $p$ and $b_j$ operator commute (cf. \eqref{eq operator comm of self-adjoint}). It follows that $p$ and $b = b_1 + i b_2$ operator commute too.
\end{proof}

We can now establish a generalization of the result proved by Arzikulov in Theorem \ref{t AJBW-} in the line of Rickart's original result. 

\begin{theorem}\label{t weakly Rickart JBstar algebras contain an abundant set of projections} Every weakly Rickart JB$^*$-algebra is generated by its projections.
\end{theorem}

\begin{proof} We can clearly reduce our argument to positive elements. Let $a$ be a positive element in a weakly Rickart JB$^*$-algebra $M$. Let $p = RP(a)$ denote the range projection of $a$ in $M$ (cf. Remark \ref{r existence and uniqueness of the range projection in wRickart JBalgebras}). It follows from Proposition \ref{p weak Rickart condition is inherited by Peirce2 projections} that $M_2(p)$ is a Rickart JB$^*$-algebra with unambiguous range projections of positive elements.\smallskip

Let $B$ be a maximal strongly associative JB$^*$-subalgebra of $M$ containing the element $a$. It follows from Lemma \ref{l strongly associativity is inherited by the range projection} that $B$ contains the range projection of every positive element $c\in B$. Therefore $B$ is a weakly Rickart associative JB$^*$-algebra, or equivalently, a commutative weakly Rickart C$^*$-algebra (cf. Propositions \ref{p sufficient and necessary conditions in terms of inner ideals} and \ref{Prop Pedersen}). Finally, it follows from Remark \ref{r wR Cstar algebras are generated by their projections} that $B$ (and hence $M$) is generated by its projections. We can also consider a maximal strongly associative JB$^*$-subalgebra $C$ of $M_2(p)$ containing $a$ and $p$. In this case $C$ is a Rickart associative JB$^*$-algebra, or equivalently, a commutative Rickart C$^*$-algebra (cf. Lemma \ref{l strongly associativity is inherited by the range projection}, Propositions \ref{p sufficient and necessary conditions in terms of inner ideals} and \ref{Prop Pedersen})   
\end{proof}

\section{Rickart JB*-triples}\label{sec: JB*-triples}

The definitions of Baer and Rickart JB$^*$-algebras introduced by Ayupov and Arzikulov and the notions of weakly Rickart and SAJBW$^*$-algebras we have developed in the previous section depend extremely on the existence of a cone of positive elements. This is a handicap if we want to work on the wider setting of JB$^*$-triples, where the existence of a cone of positive elements is, in general, impossible. \smallskip 

Furthermore, projections make no sense in the wider setting of JB$^*$-triples; and the role of projections is in general, played by tripotents. As in the original study by Rickart, our aim is to find an appropriate notion, in terms of ortogonal annihilators, local order and range tripotents, to assure that a JB$^*$-triple satisfying this property contains sufficiently many tripotents.\smallskip

The characterizations of (weakly) Rickart C$^*$-algebras established in section \ref{sec: Rickart Cstar-algebras} (see Propositions \ref{p sufficient and necessary conditions in terms of inner ideals without order} and \ref{p sufficient and necessary conditions in terms of inner ideals without order but local order}) offer a perspective which allows us to consider these notions in the wider setting of JB$^*$-triples.\smallskip

\begin{definition}\label{def Rickart and weakly Rickart JBstar triples} Let $E$ be a JB$^*$-triple. \begin{enumerate}[$\checkmark$]\item $E$ is called a SAJBW$^*$-triple if for any $x,y\in E$ with $x\perp y$, there exists a tripotent $e\in E$ satisfying $x\in E_2(e)$ and $y \in E_0(e)$.

\item  $E$ is a weakly Rickart (wR) JB$^*$-triple if given $x\in E$ and an inner ideal $J\subseteq E$ with $I = E(x)\perp J $, there exists a tripotent $e$ in $E$ such that $I\subseteq E_2(e)$ and $J\subseteq E_0(e)$.

\item $E$ is a weakly order-Rickart (woR) JB$^*$-triple if given $x\in E$ and an inner ideal $J\subseteq E$ with $I = E(x)\perp J $, there exists a tripotent $e$ in $E$ such that $x$ is positive in $E_2(e)$, and $J\subseteq E_0(e)$.

\item $E$ is called a Rickart JB$^*$-triple if it is weakly Rickart and admits a unitary element. 
\end{enumerate}
\end{definition}

For a JB$^*$-triple $E$, the following implications hold: $E$ is a Rickart JB$^*$-triple $\Rightarrow$ $E$ is a wR JB$^*$-triple, and $E$ is a woR JB$^*$-triples $\Rightarrow$ $E$ is a wR JB$^*$-triple.   \smallskip

Let $A$ be a C$^*$-algebra. It follows from Proposition \ref{p sufficient and necessary conditions in terms of inner ideals without order} that $A$ is a Rickart or a weakly Rickart C$^*$-algebra if and only if it is a Rickart or a weakly Rickart JB$^*$-triple, respectively. Furthermore, Propositions \ref{p sufficient and necessary conditions in terms of inner ideals without order but local order} and \ref{p sufficient and necessary conditions in terms of inner ideals without order} prove that a C$^*$-algebra is a wR JB$^*$-triple if and only it is a woR JB$^*$-triple.  So, our definition is consistent with the previous notions. We do not know if $A$ being a SAW$^*$-algebra implies that $A$ is a SAJBW$^*$-triple. For the reciprocal, suppose that $A$ is a SAJBW$^*$-triple. Fix two positive elements $x,y\in A$ with $x y =0,$ by hypothesis there exists a partial isometry $e$ with $x\in A_2(e)$ and $y\in A_0(e)$. Since $x = e e^* x e^* e$ and $x\geq 0$, it can be shown that $x = ee^* x = x ee^* = e^*e x = x e^*e$. Similarly, $ee^* y = y ee^* = y e^*e = ee^* y =0$. Therefore $A$ is a SAW$^*$-algebra.\smallskip

The examples provided in \cite{Rick46,BerBaerStarRings,PedSAW} show that, even in the category of abelian C$^*$-algebras, the classes of SAJBW$^*$-triples, weakly Rickart JB$^*$-triples and  weakly Rickart JB$^*$-triples are mutually different. \smallskip

In the setting of JB$^*$-algebras we do not know if there is a relation between being a Rickart or a weakly Rickart JB$^*$-algebra seen in section \ref{subsec: algebraic Jordan Rickart and Baer} and the corresponding notion as JB$^*$-triple. The lacking of polar decompositions, makes invalid the natural arguments. What we can prove is the following connection between JB$^*$-algebras which are woR JB$^*$-triples and weakly Rickart JB$^*$-algebras. 

\begin{proposition}\label{p woR JBstar triple implies wR JBstar algebra for Peirce-2} Let $M$ be a JB$^*$-algebra which is a woR JB$^*$-triple, then $M$ is a weakly Rickart JB$^*$-algebra. Actually, it suffices to assume that every positive element $a$ in $M$ admits a range tripotent $R(a)$ in $M$, and in such a case the range tripotent of $a$ in $M$ is precisely the range projection of $a$ in $M$ as weakly Rickart JB$^*$-algebra. 
\end{proposition}

Before presenting the proof, we establish a result proving the existence of range tripotents for elements in woR JB$^*$-triples. 

\begin{lemma}\label{l existence of range tripotents in woRtriples} Let $E$ be a JB$^*$-triple. Then the following statements hold:
\begin{enumerate}[$(a)$]
\item If $a$ is an element in $E$ and $e,v$ are two tripotents in $E$ such that $a$ is positive in $E_2(e)$ and in $E_2(v)$ with $\{a\}^{\perp} = E_0(e)$, then $e\leq v;$
\item Let us assume that $E$ is a woR JB$^*$-triple. Then for each element $a$ in $E$ there exists a unique tripotent $e\in E$ satisfying that $a$ is positive in $E_2(e)$ and $\{a\}^{\perp} = E_0(e)$.
\end{enumerate}
\end{lemma}

\begin{proof} $(a)$ Let $e$ and $v$ be tripotents in $E$ satisfying the properties in the statement. Let $r_{{E^{**}}}(a)$ denote the range tripotent of $a$ in $E^{**}$. Since $a$ is positive in $E_2(e)\subseteq E_2^{**}(e),$ it follows that $a$ is positive in the JBW$^*$-algebra $E^{**}_2(e),$ and hence $r_{{E^{**}}}(a)\leq e$ as tripotents in $E^{**}.$ Therefore $e = r_{{E^{**}}}(a) + (e-r_{{E^{**}}}(a))$ with $r_{{E^{**}}}(a)\perp (e-r_{{E^{**}}}(a))$, and hence $\{a,a,e\} = \{ a, a, r_{{E^{**}}}(a)\}.$ Similarly,  $\{a,a,v\} = \{ a, a, r_{{E^{**}}}(a)\}.$  It then follows that the triple product $\{a,a,e-v\} = 0$ in $E$, or equivalently, $a\perp (e-v)$, that is, $e-v\in \{a\}^{\perp}$. The assumptions on $e$ imply that $e -\{e,e,v\} = \{e,e,e-v\}=0,$ or equivalently, $\{e,e,v\} = e$.  Lemma 1.6 or Corollary 1.7 in \cite{FriRu85} implies that $v\geq e$.\smallskip

$(b)$ Let $e$ and $v$ satisfying the hypotheses in $(b)$ (both exist by the assumptions on $E$). It follows from $(a)$ that $e\leq v$ and $v\leq e$. Therfore $e=v$ as claimed. 
\end{proof}

Let $a$ be an element in a woR JB$^*$-triple $E$. The unique tripotent $e$ given by Lemma \ref{l existence of range tripotents in woRtriples} is called the \emph{range tripotent} of $a$ in $E$, and will be denoted by $R_{_E} (a)$. It follows from Lemma \ref{l existence of range tripotents in woRtriples}$(a)$ that $R_{_E}(a)$ is the smallest tripotent $e$ in $E$ sastisfying that $a$ is positive in the unital JB$^*$-algebra $E_2(e)$. \smallskip

Let us briefly recall that for each self-adjoint element $h$ in a JB$^*$-algebra $M,$ the mapping $U_{h}$ is positive on $M$, that is, it maps positive elements to positive elements \cite[Proposition 3.3.6]{HOS}.

\begin{proof}[Proof of Proposition \ref{p woR JBstar triple implies wR JBstar algebra for Peirce-2}] Let us fix a positive element $a$ in $M$. Let $e= R_{_E} (a)$ denote the range tripotent of $a$ in $E$. Since the involution on $M$ is a conjugate linear triple automorphism on $M$ we have $0\leq a = a^*$ in $M_2(e^*)$ and $$\{a\}^{\perp} = \{a^*\}^{\perp} = \left(\{a\}^{\perp} \right)^{*} = \left(M_0(e)\right)^* = M_0(e^*),$$ witnessing that $e^*$ satisfies the properties of the range tripotent of $a$ in $M,$ and by the uniqueness of this element $e= e^*$. That is, $e$ is self-adjoint tripotent in $M$, and thus, by the local Gelfand theory, $e= p-q,$ where $p$ and $q$ are two orthogonal projections in $M$.\smallskip

It follows from the properties of the range tripotent $e= p-q$ that $0\leq a$ in $M_2(e).$ Since $0\leq -q\leq e$ in $M_2(e)$, the element $-q$ is a projection in $M_2(e).$ Therefore, having in mind that, by Kaup's theorem, the triple product on $M_2(e)$ is uniquely given by the restriction of the triple product of $M$ and by the JB$^*$-structure of $M_2(e)$,  the element $$U_{-q}^{M_2(e)}(a) = \{-q, a^{*_e} ,-q\} = \{-q, a,-q\}= \{q, a,q\} = U_{q}(a)$$ is positive in $M_2(e)$ (cf. \cite[Proposition 3.3.6]{HOS}), and in $M_2(-q).$ Since, $M_2(-q) = M_2(q)$ with $\left(M_2(-q)\right)_{sa} = \left(M_2(q)\right)_{sa}$ we deduce the existence of $y\in \left(M_2(-q)\right)_{sa} = \left(M_2(q)\right)_{sa}\subseteq M_{sa}$ such that $$U_{q}(a) = y\circ_{-q} y = \{y, -q, y\} = -\{y,q,y\} = - U_y (q),$$ which implies that $U_{q}(a)$ is a negative element in $M$. \smallskip

On the other hand, since $a$ is positive in $M$ and $q$ is a projection, the element $U_q(a)$ must be positive in $M$ \cite[Proposition 3.3.6]{HOS}, which combined with the previous conclusion leads to $U_q (a) =0.$  It follows from the first statement in Lemma \ref{l positive sets and annihilators} that $q\in \{a\}^{\perp_q}\cap M_{sa} = \{a\}^{\perp} \cap M_{sa}$, that is, $q\perp a$. The properties of the range tripotent imply that $q\in M_0(e) = M_0(p-q),$ and thus $q\perp (p-q),$ and so $q=0.$\smallskip

We have therefore shown that the range tripotent $e=R_{_M}(a)$ of $a$ in $M$ is a projection in this JB$^*$-algebra. It can be easily checked that $e\circ a = \{e,e,a\} =a$ and for each $z\in M_{sa}$ with $U_z (a) =0$ we have $p\circ z =0$ (cf. Lemma \ref{l positive sets and annihilators}), that is $M$ is a weakly Rickart JB$^*$-algebra. 
\end{proof}

An element $u$ in a unital JB$^*$-algebra $M$ is called a \emph{unitary} if it is invertible with inverse $u^*$. In the setting of JB$^*$-triples, the word \emph{unitary} is applied to those elements $u$ such that $L(u,u)$ is the idenity mapping. Clearly, every unitary $u$ in a JB$^*$-triple $E$ is a tripotent with $E_2(u)=E$ --this is actually a characterization. There is no ambiguity in case that a unital JB$^*$-algebra $M$ is regarded as a JB$^*$-triple because both notions are equivalent \cite[Proposition 4.3]{BraKaUp78}. \smallskip




Our next result is a strengthened version of  Proposition \ref{l existence of range tripotents in woRtriples}. We recall first that for each tripotent $e$ in a JB$^*$-triple $E$ and each unitary complex number $\lambda,$ the mapping \begin{equation}\label{eq Slambda is a triple automorphism} S_{\lambda}(e) = \lambda^2 P_2(e) + \lambda P_1(e) + P_0(e)\end{equation} is a triple automorphism on $E$ \cite[Lemma 1.1]{FriRu85}. It can be easily deduced from this fact that the mapping \begin{equation}\label{eq Rlambda is a triple automorphism} R_{\lambda}(e) = P_2(e) + \lambda P_1(e) + \lambda^2 P_0(e)\end{equation} also is a triple automorphism on $E$.

\begin{proposition}\label{p woR implies wR for each Peirce-2} Let $E$ be a woR JB$^*$-triple. Then for each tripotent $e\in E,$ the Peirce-2 subspace $E_2(e)$ is a Rickart JB$^*$-algebra.
\end{proposition}

\begin{proof} Having in mind Proposition \ref{p woR JBstar triple implies wR JBstar algebra for Peirce-2} and Lemma \ref{l weak Rickart and weak Baer}$(a)$, it suffices to show that each positive element $a$ in $E_2(e)$ admits a range tripotent in $E_2(e).$  Let $v = R_{_E}(a)$ be the range tripotent of $a$ in $E$. Let $S_{-1} = S_{-1} (e) = P_2(e) - P_1(e) + P_0(e)$ the triple automorphism on $E$ given in \eqref{eq Slambda is a triple automorphism}. Let us observe that $S_{-1} (a) =a$ because $a\in E_2(e).$

Since $a$ is positive in $E_2(v)$ with $\{a\}_{E}^{\perp} = E_0(v),$ we deduce that $a = S_{-1} (a)$ is positive in $E_2(S_{-1}(v))$ with $$\{a\}_{E}^{\perp} = \{S_{-1}(a) \}_{E}^{\perp} = S_{-1} \left(\{a\}_{E}^{\perp}\right) = S_{-1} \left(E_0(v)\right) = E_0(S_{-1}(v)).$$ That is, $S_{-1} (v)$ satisfies the properties of the range tripotent for $a$, and hence it follows from its uniqueness that $v = S_{-1} (v) = P_2(e) (v) -P_1(e) (v) + P_0(e) (v).$ This equality proves that $v = P_2(e) (v) + P_0(e) (v),$ where $P_2(e) (v) $ and $P_0(e) (v)$ are two orthogonal tripotents in $E.$ \smallskip

If in the previous argument we replace $S_{-1} (e)$ with $R_{i} (e)$, and we apply it to $v = P_2(e) (v) + P_0(e) (v),$ we derive that $v = R_{i} (e) (v) = P_2(e) (v) - P_0(e) (v)$, witnessing that $v = P_2(e) (v).$ Now, it can be easily seen that $v = P_2(e) (v)\in E_2(e)$ satisfies the properties of the range tripotent for $a$ in $E_2(e)$ (and in $E$). This concludes the proof.    
\end{proof}

We can now establish the result which has motivated our study. We shall see that every woR JB$^*$-triple contains an abundant collection of tripotents. 

\begin{theorem}\label{t woR JBstriples are generated by its tripotents} Every weakly order Rickart JB$^*$-triple is generated by its tripotents.
\end{theorem}

\begin{proof} Let $a$ be an element in a woR JB$^*$-triple $E$. Let $e= R_{_E}(a)$ be the range tripotent of $a$ in $E$. Proposition \ref{p weak Rickart condition is inherited by Peirce2 projections} assures that $E_2(e)$ is Rickart JB$^*$-algebra. By construction, $a$ is a positive element in $E_2(e)$, and hence Theorem \ref{t weakly Rickart JBstar algebras contain an abundant set of projections} implies that $a$ can be approximated in norm by finite linear combinations of projections in $E_2(e).$ The proof concludes by just observing that since $E_2(e)$ is a JB$^*$-subtriple of $E$, every projection in $E_2(e)$ is a tripotent in $E$. 
\end{proof}

\section{von Neumann regularity}\label{sec: von Neumann regularity}

Regular elements in the sense of von Neumann have been intensively studied in the associative setting of C$^*$-algebras (cf. \cite{HarMb92, HarMb93, BP95} and \cite[\S 3]{Rick46}) as well as in the wider setting of JB$^*$-triples (see \cite{FerGarSanSi92,FerGarSanSi94, Ka96, BurKaMoPeRa, BurMarMorPe16} and \cite{JamPeSiddTah2016}).\smallskip

Motivated by the study conducted by Rickart on von Neumann regular elements in $B_p^*$-algebras (now called Rickart C$^*$-algebras) in \cite[\S 3]{Rick46}, we devote this section to explore von Neumann regular elements in woR JB$^*$-triples.\smallskip

An element $a$ in a JB$^*$-triple $E$ is called \emph{von Neumann regular} if and only if there exists $b\in E$ such that $Q(a)b =a,$ $Q(b)a =b$ and $[Q(a),Q(b)]:=Q(a)\,Q(b) - Q(b)\, Q(a)=0$ (cf.  \cite[Lemma 4.1]{Ka96} or \cite{FerGarSanSi92,FerGarSanSi94, BurKaMoPeRa}). The element $b \in E$ satisfying the previous properties is unique and is called the \emph{generalized inverse} of $a$ in $E$  (denoted by $a^{\dag}$). However, there exist von Neumann regular elements $a\in E$, for which we can find many elements $c$ in $E$ such that $Q(a)c =a$. \smallskip

Several useful characterizations of von Neumann regular elements in JB$^*$-triples can be found in \cite{FerGarSanSi92,FerGarSanSi94, Ka96, BurKaMoPeRa}. For our purposes here, we recall that an element $a$ in a JB$^*$-triple $E$, whose range tripotent in $E^{**}$ is denoted by $r_{{E^{**}}}(a)= r(a),$ we know that $a$ is von Neumann regular if, and only if, $r(a)\in E$ and $a$ is positive and invertible in the unital JB$^*$-algebra $E_2 (r(a)),$ and in such a case $a^{\dag}$ is precisely the inverse of $a$ in $E_2(r(a))$ (cf. \cite[\S 2, pages 191 and 192]{BurKaMoPeRa}). It is further known that in this case $L(a,a^{\dag}) = L(a^{\dag},a) = L (r(a),r(a))$ (see \cite[\S 2, page 192]{BurKaMoPeRa} and \cite[Lemma 3.2]{Ka2001}).\smallskip

The next lemma goes in the line of \cite[Lemma 2.2]{JamPeSiddTah2015} and \cite[Theorem 3.2]{Rick46}.

\begin{lemma}\label{l distance smaller than one from a tripotent} Let $e$ be a tripotent in a JB$^*$-triple $E$. The following statements hold:
\begin{enumerate}[$(a)$]
\item Every invertible element $a$ in the unital JB$^*$-algebra $E_2(e)$ is von Neumann regular in $E$ with $r_{{E^{**}}}(a)$ being a unitary element in $E_2(e).$
\item Suppose that $x$ is an element in $E$ with $\|e-x\|<1$. Then $Q(e)(x)$ and $P_2(e) (x)$ are von Neumann regular elements whose range tripotents {\rm(}i.e. $r(Q(e)(x))$ and $r(P_2(e)(x))$, respectively{\rm)} in $E^{**}$ belong to $E_2(e)$ and are unitaries in the latter JB$^*$-algebra. Moreover, $r(Q(e)(x))$ and $r(P_2(e)(x))$ satisfy the properties of the range tripotent in a woR JB$^*$-triple for the elements $Q(e) (x)$ and $P_2(e) (x)$, respectively. The latter conclusion holds for the range tripotent in $E^{**}$ of any invertible element $a\in E_2(e)$.
\end{enumerate}
\end{lemma}

\begin{proof}$(a)$ The statement is essentially proved in \cite[Remark 2.3]{JamPeSiddTah2015}. Namely, if $a$ is invertible in $E_2(e),$ the just quoted remark assures that the range tripotent $r= r_{_{E^{**}_2(e)}} (a)$ of $a$ in the bidual of $E_2(e)$ is a unitary element in $E_2(e).$ It is clear that $r$ must be also the range tripotent of $a$ in $E^{**}$ and belongs to $E$. It follows from the characterization of von Neumann regular elements from \cite{BurKaMoPeRa}, seen before this lemma, that $a$ is von Neumann regular in $E$.\smallskip

$(b)$ Since $\|e-x\|<1$ and $Q(e)$ and $P_2(e)$ are non-expansive mappings fixing the element $e$, we get $\|e-Q(e)(x)\|, \|e-P_2(e)(x)\|<1$. Having in mind that $E_2(e)$ is a unital JB$^*$-algebra with unit $e$ and $Q(e) (x), P_2(e) (x)\in E_2(e)$, we deduce that these two elements are invertible in $E_2(e).$ The first part of the statement now follows from $(a)$.\smallskip 

We shall only prove the last statement for $P_2(e) (x).$ To simplify the notation, let $r= r(P_2(e)(x))\in E_2(e)$ denote the range tripotent of $P_2(e)(x)$. Clearly, $P_2(e)(x)$ is positive in $E_2(r)$ (let us note that $E_2(r) = E_2(e)$ as sets because $r$ is a unitary in $E_2(e)$). Finally, it follows Lemma 3.2 in \cite{BurGarPe11StudiaTriples} that $\{P_2(e)(x)\}^{\perp} = E_0 (r)$, which concludes the argument.
\end{proof}

The next result is a triple version of \cite[Theorem 3.3]{Rick46}.

\begin{proposition}\label{p Rickart thm33} Let $E$ be a woR JB$^*$-triple. Suppose that $a$ is a von Neumann regular element in $E$. Then the range tripotent of $a$ in $E$ as woR JB$^*$-triple coincides with the range tripotent of $a$ in $E^{**}$ {\rm(}and in $E${\rm)}, that is $R(a) = r_{{E^{**}}}(a)$. Furthermore $a^{\dag}\in E_2(R(a))$ is the inverse of $a$ in $E_2(R(a))$ and $R(a^{\dag}) = R(a).$  
\end{proposition}

\begin{proof} We know from Lemma \ref{l distance smaller than one from a tripotent}$(b)$ that the range tripotent $r(a)$ satisfies the properties of the range tripotent of $a$ in the definition of woR JB$^*$-triple. Then the uniqueness of $R(a)$ (see Lemma \ref{l existence of range tripotents in woRtriples}$(b)$) implies that $R(a) = r(a).$\smallskip

It is known that $r=r(a)=R(a)$ and $a^{\dag}$ both belong to the JB$^*$-subtriple of $E$ generated by $a$ (cf. \cite[Lemma 3.2]{Ka2001}), and hence $a^{\dag} \in E_2(R(a)).$ Finally, we know from the properties of the generalized inverse that $a^{\dag}$ is the inverse of $a$ in $E_2(r)$.
\end{proof}

As we have seen in subsection \ref{subsec: background}, for each element $a$ in a JB$^*$-triple $E$, its triple spectrum $\Omega_a\subseteq [0,\|a\|]$ can be employed to identify the JB$^*$-subtriple, $E_a$, of $E$ generated by $a$ with the commutative C$^*$-algebra $C_0(\Omega_{a}),$ and under this identification $a$ corresponds to the continuous function given by the embedding of $\Omega_{a}$ into $\mathbb{C}$ (cf. \cite[Corollary 1.15]{Ka} and \cite[Lemma 3.2]{Ka96}). The triple spectrum $\Omega_a$ does not change when computed with respect to any JB$^*$-subtriple $F$ of $E$ containing the element $a$ \cite[Proposition 3.5$(vi)$]{Ka96}.  It is further known that $a$ is von Neumann regular if and only if $0\notin \Omega_a$ (cf. \cite[Lemma 4.1]{Ka96}). In particular if $F$ is a JB$^*$-subtriple of a JB$^*$-triple $E$, then an element $a\in F$ is von Neumann regular in $F$ if and only if it is von Neumann regular in $E$. Furthermore, if $a\in E$ is von Neumann regular, then $a^{\dag}$ and $r(a)$ both belong to the JB$^*$-subtriple of $E$ generated by $a$.\smallskip

Our next goal is a triple version of \cite[Theorem 3.13]{Rick46} and a refinement of Theorem \ref{t woR JBstriples are generated by its tripotents}.

\begin{proposition}\label{p approximation by regular elements} Let $E$ be a woR JB$^*$-triple. Suppose $a$ is an element in $E$ whose range tripotent is $R(a).$ Then for each $\varepsilon>0$ there exists a tripotent $e_{\varepsilon}\in E$ and an element $b$ in the JB$^*$-subtriple of $E$ generated by $a$ satisfying $e_{\varepsilon} \leq R(a),$ $\{b,R(a),b\}=a,$ $ \{b, e_{\varepsilon}, b\}$ is von Neumann regular and $\|a - \{b, e_{\varepsilon}, b\} \| < \varepsilon.$
\end{proposition}

\begin{proof} Proposition \ref{p woR implies wR for each Peirce-2} assures that $E_2(R(a))$ is Rickart JB$^*$-algebra. By definition, $a$ is positive in $E_2(R(a)).$ Let $C$ be a maximal strongly associative JB$^*$-subalgebra of $E_2(R(a))$ containing $a$. Lemma \ref{l strongly associativity is inherited by the range projection} implies that $C$ is a Rickart JB$^*$-algebra. Therefore $C$ is a commutative Rickart C$^*$-algebra whose product and involution will be denoted by $\cdot$ and $*$, respectively --observe that $*$ coincides with ${*_{R(a)}}$. \smallskip

Given $\varepsilon>0,$ having in mind that $C$ is a commutative C$^*$-algebra, Theorem 3.13 in \cite{Rick46} proves the existence of a projection $e_{\varepsilon}\in C$ satisfying $e_{\varepsilon}\leq R(a)$, $e_{\varepsilon}\cdot a= \{e_{\varepsilon}, a, e_{\varepsilon}\} = P_2(e_{\varepsilon}) (a)$ is von Neumann regular in $C$ and $\|a-P_2(e_{\varepsilon}) (a)\|<\varepsilon.$\smallskip 

As observed in \cite[comments after Theorem 2.1]{GarPer2}, since $a$ is a positive in $E_2(R(a))$ (and in $C$), the JB$^*$-subtriple $E_a$ of $E_2(R(a))$ (and of $C$) generated by $a$ coincides with the JB$^*$-subalgebra that $a$ generates. Therefore the square root of $a$ in $C$ lies in $E_a$. Let $b\in E_a$ denote the square root of $a$ in $C$. By applying that $C$ is a commutative C$^*$-algebra, it can be deduced that $\{b,e_{\varepsilon}, b\} = (b\cdot b) \cdot e_{\varepsilon} = a\cdot e_{\varepsilon}$ is von Neumann regular in $C$. Clearly,  $\{b,R(a),b\}=a.$\smallskip

Finally, since $C$ is a JB$^*$-subtriple of $E,$ the element $e_{\varepsilon}$ is a tripotent in $E$ with  $e_{\varepsilon}\leq R(a)$, and $\{b, e_{\varepsilon}, b\} $ is von Neumann regular in $E$ and $\|a- \{b, e_{\varepsilon}, b\} \|< \varepsilon.$
\end{proof}

We can now prove that every inner ideal in a woR JB$^*$-triple $E$ contains an abundant collection of von Neumann regular elements.

\begin{theorem}\label{t inner ideals of woR JBstar triples contain a norm dense subset of regular elements} Let $I$ be an inner ideal of a woR JB$^*$-triple $E$. Then the von Neumann regular elements of $I$ are dense in $I$. Each von Neumann regular element $x$ in $I$ is contained in $E_2(R(x)),$ where $R(x)\in I$ and $E_2(R(x))$ is a Rickart JB$^*$-algebra. Furthermore, if $I\neq \{0\},$ then $I$ contains a non-zero tripotent, actually $I$ contains the generalized inverse and the range tripotent of each non-zero element in $I$.  
\end{theorem}

\begin{proof} Let us fix $a\in I$. Proposition \ref{p approximation by regular elements} proves that we can approximate $a$ in norm by von Neumann regular elements of the form $\{b, e, b\},$ where $e\in E$ is a tripotent satisfying $e \leq R(a)$ and $b\in E_a.$ Having in mind that $I$ is an inner ideal we deduce that $E_a\subseteq I,$ and $\{b,e,b\}\in I,$ which concludes the proof of the first statement. The second statement is a consequence of Propositions \ref{p Rickart thm33} and \ref{p woR implies wR for each Peirce-2}.\smallskip

Take now $a\in I\backslash\{0\}.$ In this case $E_a\subseteq E(a).$ By the conclusion in the first paragraph, we can approximate $a$ in norm by a sequence $(a_n)_n$ of non-zero von Neumann regular elements in $I$. It follows from Proposition \ref{p Rickart thm33} that the range tripotent of each $a_n$ in $E$, $R(a_n),$ coincides with its range tripotent in $E^{**}$ and by the theory on von Neuman regular elements $a_n^{\dag}, R(a_n)\in E_{a_n} \subseteq E(a_n)\subseteq I,$ which concludes the proof. 
\end{proof}

\noindent\textbf{Acknowledgements} First and third authors are partially supported by the Spanish Ministry of Science, Innovation and Universities (MICINN) and European Regional Development Fund project no. PGC2018-093332-B-I00 and Junta de Andaluc{\'i}a grants number A-FQM-242-UGR18 and FQM375. Second author is partially supported by NSF of China (12171251). Third author is also supported by the IMAG--Mar{\'i}a de Maeztu grant CEX2020-001105-M / AEI / 10.13039 / 501100011033. The fourth author is supported by a grant from the
``Research Center of the Female Scientific and Medical Colleges'', Deanship of Scientific Research, King Saud University.

\end{document}